\documentclass[11pt,a4paper,reqno]{amsart}

\usepackage{fullpage,amssymb}

\usepackage[nobysame,alphabetic,initials,msc-links,lite]{amsrefs}

\usepackage{mathscinet}

\usepackage[all]{xy}

\DefineSimpleKey{bib}{how}
\renewcommand{\PrintDOI}[1]{%
  \href{http://dx.doi.org/#1}{{\tt DOI:#1}}%
}
\renewcommand{\eprint}[1]{#1}
\BibSpec{book}{%
    +{}  {\PrintPrimary}                {transition}
    +{.} { \PrintDate}                  {date}
    +{.} { \textit}                     {title}
    +{.} { }                            {part}
    +{:} { \textit}                     {subtitle}
    +{,} { \PrintEdition}               {edition}
    +{}  { \PrintEditorsB}              {editor}
    +{,} { \PrintTranslatorsC}          {translator}
    +{,} { \PrintContributions}         {contribution}
    +{,} { }                            {series}
    +{,} { \voltext}                    {volume}
    +{,} { }                            {publisher}
    +{,} { }                            {organization}
    +{,} { }                            {address}
    +{,} { }                            {status}
    +{,} { \PrintDOI}                   {doi}
    +{,} { \PrintISBNs}                 {isbn}
    +{}  { \parenthesize}               {language}
    +{}  { \PrintTranslation}           {translation}
    +{;} { \PrintReprint}               {reprint}
    +{.} { }                            {note}
    +{.} {}                             {transition}
    +{}  {\SentenceSpace \PrintReviews} {review}
}
\BibSpec{article}{%
    +{}  {\PrintAuthors}                {author}
    +{,} { \textit}                     {title}
    +{.} { }                            {part}
    +{:} { \textit}                     {subtitle}
    +{,} { \PrintContributions}         {contribution}
    +{.} { \PrintPartials}              {partial}
    +{,} { }                            {journal}
    +{}  { \textbf}                     {volume}
    +{}  { \PrintDatePV}                {date}
    +{,} { \issuetext}                  {number}
    +{,} { \eprintpages}                {pages}
    +{,} { }                            {status}
    +{,} { \PrintDOI}                   {doi}
    +{,} { \eprint}        {eprint}
    +{}  { \parenthesize}               {language}
    +{}  { \PrintTranslation}           {translation}
    +{;} { \PrintReprint}               {reprint}
    +{.} { }                            {note}
    +{.} {}                             {transition}
    +{}  {\SentenceSpace \PrintReviews} {review}
}
\BibSpec{collection.article}{%
    +{}  {\PrintAuthors}                {author}
    +{,} { \textit}                     {title}
    +{.} { }                            {part}
    +{:} { \textit}                     {subtitle}
    +{,} { \PrintContributions}         {contribution}
    +{,} { \PrintConference}            {conference}
    +{}  {\PrintBook}                   {book}
    +{,} { }                            {booktitle}
    +{,} { \PrintDateB}                 {date}
    +{,} { pp.~}                        {pages}
    +{,} { }                            {publisher}
    +{,} { }                            {organization}
    +{,} { }                            {address}
    +{,} { }                            {status}
    +{,} { \PrintDOI}                   {doi}
    +{,} { \eprint}        {eprint}
    +{}  { \parenthesize}               {language}
    +{}  { \PrintTranslation}           {translation}
    +{;} { \PrintReprint}               {reprint}
    +{.} { }                            {note}
    +{.} {}                             {transition}
    +{}  {\SentenceSpace \PrintReviews} {review}
}
\BibSpec{misc}{%
  +{}{\PrintAuthors}  {author}
  +{,}{ \textit}      {title}
  +{.}{ }             {how}
  +{}{ \parenthesize} {date}
  +{,} { available at \eprint}        {eprint}
  +{,}{ available at \url}{url}
  +{,}{ }             {note}
  +{.}{}              {transition}
}

\newtheorem{theorem}{Theorem}[section]
\newtheorem{lemma}[theorem]{Lemma}
\newtheorem{proposition}[theorem]{Proposition}
\newtheorem{corollary}[theorem]{Corollary}

\theoremstyle{definition}
\newtheorem{definition}[theorem]{Definition}

\theoremstyle{remark}
\newtheorem{rem}[theorem]{Remark}
\newtheorem{remark}[theorem]{Remark}

\newtheorem{example}[theorem]{Example}

\newcommand\pf{\begin{proof}}
\newcommand\epf{\end{proof}}
\newcommand\bp{\begin{proof}}
\newcommand\ep{\end{proof}}

\usepackage{dsfont,stmaryrd}
\usepackage{mathdots}
\usepackage{wasysym}
\newcommand{\circt}{\mathbin{\ooalign{$\ocircle$\cr\hidewidth\raise-.15ex\hbox{$\scriptstyle\top\mkern2.1mu$}\cr}}} % Woronowicz style tensor product, USUAL SIZE
\newcommand{\smCirct}{\mathbin{\ooalign{$\scriptstyle\ocircle$\cr\hidewidth\raise-.12ex\hbox{$\scriptscriptstyle\top\mkern1mu$}\cr}}}  % Woronowicz style tensor product, SCRIPT SIZE

\newcommand{\C}{\mathbb{C}}
\newcommand{\R}{\mathbb{R}}
\newcommand{\T}{\mathbb{T}}
\newcommand{\ZZ}{\mathbb{Z}}
\newcommand{\N}{\mathbb{N}}

\newcommand{\A}{{\mathcal A}}

\newcommand{\BB}{{\mathcal B}}
\newcommand{\CC}{{\mathcal C}}
\newcommand{\DD}{\mathcal{D}}
\newcommand{\PP}{{\mathcal P}}
\newcommand{\U}{\mathcal{U}}
\newcommand{\eps}{\varepsilon}

\newcommand{\Dhat}{\hat\Delta}
\newcommand{\un}{\mathds{1}}
\newcommand{\abel}{\mathrm{ab}}

\newcommand{\Pol}{\mathcal{O}}
\newcommand{\couni}{\varepsilon}
\newcommand{\co}{\mathrm{co}}
\newcommand{\reg}{\mathrm{reg}}

\newcommand{\K}{\mathrm{K}}
\newcommand{\GL}{\mathrm{GL}}
\newcommand{\SL}{\mathrm{SL}}
\newcommand{\Hpf}{\mathcal{H}}
\newcommand{\SU}{\mathord{\mathrm{SU}}}

\DeclareMathOperator{\Hilb}{Hilb}

\DeclareMathOperator{\Aut}{Aut}
\DeclareMathOperator{\Hom}{Hom}

\DeclareMathOperator{\Reg}{Reg}
\DeclareMathOperator{\End}{End}
\DeclareMathOperator{\Ker}{Ker}

\DeclareMathOperator{\Rep}{Rep}
\DeclareMathOperator{\Corep}{Corep}

\DeclareMathOperator{\Ch}{Ch}
\DeclareMathOperator{\Id}{Id}

\DeclareMathOperator{\ract}{\lhd}
\DeclareMathOperator{\lact}{\rhd}
\DeclareMathOperator{\Vect}{Vect}
\DeclareMathOperator{\ev}{ev}

\DeclareMathOperator{\St}{St}

\DeclareMathOperator{\tr}{tr}
\DeclareMathOperator{\Mat}{Mat}

\numberwithin{equation}{section}

\title{Graded twisting of categories and quantum groups by group actions}

\author[J. Bichon]{Julien Bichon}

\email{Julien.Bichon@math.univ-bpclermont.fr}

\address{Laboratoire de Math\'{e}matiques, Universit\'{e} Blaise Pascal, Campus universitaire des C\'{e}zeaux, 3 place Vasarely, 63178 Aubi\`{e}re Cedex, France}

\author[S. Neshveyev]{Sergey Neshveyev}

\email{sergeyn@math.uio.no}

\address{Department of Mathematics, University of Oslo, P.O. Box 1053
Blindern, NO-0316 Oslo, Norway}

\thanks{The research leading to these results has received funding
from the European Research Council under the European Union's Seventh
Framework Programme (FP/2007-2013)
/ ERC Grant Agreement no. 307663%--NCGQG
}

\author[M. Yamashita]{Makoto Yamashita}

\email{yamashita.makoto@ocha.ac.jp}

\address{Department of Mathematics, Ochanomizu University, Otsuka
2-1-1, Bunkyo, 112-8610 Tokyo, Japan}

\thanks{Supported by JSPS KAKENHI Grant Number 25800058}

\begin{document}

\begin{abstract}
Given a Hopf algebra $A$ graded by a discrete group together with an action of the same group preserving the grading, we define a new Hopf algebra, which we call the graded twisting of $A$. If the action is by adjoint maps, this new Hopf algebra is a twist of $A$ by a pseudo-$2$-cocycle. Analogous construction can be carried out for monoidal categories. As examples we consider graded twistings of the Hopf algebras of nondegenerate bilinear forms, their free products, hyperoctahedral quantum groups and $q$-deformations of compact semisimple Lie groups. As  applications, we show that the analogues of the Kazhdan--Wenzl categories in the general semisimple case cannot be always realized as representation categories of compact quantum groups, and
for genuine compact groups, we analyze quantum subgroups of the new twisted compact quantum groups, providing a full description when the twisting group is cyclic of prime order.
\end{abstract}

\date{June 30, 2015; minor revision April 27, 2016}

\maketitle

\section*{Introduction}
\label{sec:intro}

Various forms of twisting constructions have appeared in the theory of
quantum groups since the famous proof of the Kohno--Drinfeld theorem by
Drinfeld~\cite{MR1047964}, where he studied twistings of quasi-Hopf algebras
and showed that the quantized universal enveloping algebras are twists of the usual universal enveloping algebras equipped with the associators defined by the Knizhnik--Zamolodchikov equation.  On the dual side of quantized function algebras, the
corresponding procedure for $2$-cocycle twisting appeared in a work of
Doi~\cite{MR1213985}, which clarified the relation between the multiparametric deformations of $\GL(n)$ by Takeuchi~\cite{MR1065785} and Artin--Schelter--Tate~\cite{MR1127037} for different parameters.

In the present paper we study twistings of Hopf algebras by a particular class of pseudo-$2$-cocycles (meaning that the new coproduct is still strictly associative) and related
generalizations.  One feature of twisting by pseudo-$2$-cocycles is that the
corresponding tensor category of corepresentations is perturbed in a very controlled
way, whereas the $2$-cocycle twisting does not change the
category at all.  Specifically, the
new Hopf algebra has the same combinatorial structure of corepresentations
in terms of fusion rules, quantum dimensions, and classical
dimensions, so that the perturbation is very mild.

The motivation for this paper stems from the previous work of the second and third authors~\cite{MR3340190} which
gave quantum group realizations of the Kazhdan--Wenzl categories~\cite{MR1237835}, which are the representation categories of $\SL_q(n)$ with nontrivial
associators given by $3$-cocycles on the chain group of the category.
It turns out that a similar construction can be carried out starting from a $\Gamma$-grading on a Hopf algebra $A$ and an action by the same discrete group $\Gamma$ preserving the grading.  If, moreover, the
automorphisms defining the action are adjoint, then the failure of lifting the action to a
homomorphism from $\Gamma$ into the group $H^1(A)$ of characters of~$A$ (with convolution product) gives rise to a pseudo-$2$-cocycle and therefore to a new associator on the monoidal category of corepresentations of $A$. By
allowing non-adjoint actions, we further
expand the scope of our twisting procedure, resulting in categories
which have different fusion rules than the original ones. The procedure can also
be formulated at the level of monoidal categories.

In more detail, our twisting consists of two quite elementary steps.
The first step is to take the \emph{crossed products} (\emph{smash products}),
which give Hopf algebras from group actions on Hopf algebras. In
the framework of monoidal categories an analogous construction was given
by Tambara~\cite{MR1815142}. The second step is to take the \emph{diagonal subalgebra} or
\emph{diagonal subcategory} inside the crossed product (see Section~\ref{sec:grtw} for precise formulations) using
the grading. The compatibility between the grading and
the action ensures that we obtain a Hopf algebra or a monoidal category as a
result.

We provide applications of our framework to two general problems in quantum group theory.

\begin{enumerate}
 \item The first is the quantum group realization problem for the Kazhdan--Wenzl type categories, mentioned above. The precise formulation is as follows: given a compact quantum group $G$ and a $3$-cocycle $c$ on the chain group of its representation category, is the twisted category $(\Rep G)^c$ (with associator given by $c$) still the representation category of a compact quantum group, or in other words, does it admit a unitary fiber functor? Using \cite{arXiv:1405.6572}, we give a necessary and sufficient condition in order that $(\Rep G)^c$ admits a dimension preserving fiber functor when $G$ is coamenable (Corollary \ref{cfib}). We then show that the analogues of the Kazhdan--Wenzl categories in the semisimple case cannot be always realized as representation categories of compact quantum groups (Corollary \ref{cor:kwnotqg}).

In case $G$ is the $q$-deformation of a compact semisimple Lie group,
our procedure for finding pseudo-$2$-cocycles boils down to a construction of cochains on the dual of the maximal torus with coboundary living on the dual of the center, which has been discussed in \citelist{\cite{MR3340190}\cite{arXiv:1405.6574}}. While this might look like a very special construction, we prove,  that for $q>0$, $q\ne1$, up to coboundaries there are no other unitary pseudo-$2$-cocycles with the corresponding associator given by a $3$-cocycle on the chain group (Theorem \ref{thm:3coc-2pseudo}).

\item The second problem is the determination of the quantum subgroups of the twisted quantum groups. We provide a full answer if our compact quantum group is obtained as the graded twisting of a genuine compact group by a cyclic group of prime order (Theorem \ref{thm:subgroups}).  This is a wide generalization of the description of the quantum subgroups of $\SU_{-1}(2)$ by Podle\'s \cite{MR1331688},  and complements the previous work of the first author and Yuncken~\cite{MR3194750}
on a similar problem for $2$-cocycle twistings. Similarly to~\cite{MR3194750}, our analysis is based on a careful study of irreducible representations of the twisted algebra of regular functions.
\end{enumerate}

\smallskip

The  paper is organized  as follows.
Section~\ref{sec:prelim} contains some background material on graded categories and Hopf algebras, group actions and crossed products.

In Section~\ref{sec:grtw} we give a detailed presentation of our twisting procedure, first for monoidal categories and then for Hopf algebras. We then consider several examples arising from Hopf algebras of nondegenerate bilinear forms and their free products.

In Section~\ref{sec:cpt-q-grps} we specialize to the case of Hopf algebras of regular functions on compact quantum groups and consider more examples coming from free orthogonal and unitary quantum groups, half-liberations of orthogonal and unitary groups, hyperoctahedral quantum groups, and $q$-deformations of compact semisimple Lie groups.
We then discuss the realization problem for the Kazhdan--Wenzl type categories.

In Section~\ref{sec:cpt-grp-twist} we further specialize to Hopf algebras of regular functions on genuine compact groups and study the problem of describing quantum subgroups of the corresponding twisted compact quantum groups.

\medskip
\paragraph{\bf Acknowledgement} The authors would like to thank the reviewer for careful reading of the manuscript, which led to elimination of several inaccuracies and to improvement of the presentation.

\section{Preliminaries}
\label{sec:prelim}

Throughout the whole paper we denote by $K$ a commutative field, and by $\Gamma$ a discrete group. We consider only vector spaces, algebras, etc., over $K$. When we talk about $*$-structures, we assume that $K=\C$.

\subsection{Graded categories}
\label{sec:gamma-gr-mon-cat}

A $\Gamma$-graded category $\CC$ is a $K$-linear category with full subcategories~$\CC_g$ for $g \in \Gamma$, such that any object $X$ in $\CC$ admits a unique (up to isomorphism) decomposition $X \simeq \oplus_{g \in \Gamma} X_g$ with $X_g \in \CC_g$ such that $X_g=0$ for all but a finite number of $g$'s, and there are no nonzero morphisms between objects in $\CC_g$ and $\CC_h$ for $g\ne h$. Let us say that the grading is \emph{full} if the subcategories $\CC_g$ are strictly full, i.e.,~closed under taking isomorphic objects. Clearly, any grading uniquely extends to a full grading, and we do not lose generality by imposing this condition.

We say that~$\CC$ is a $\Gamma$-graded monoidal category if in addition it is a monoidal $K$-linear category, such that $\un \in \CC_e$ and the monoidal structure satisfies $X \otimes Y \in \CC_{g h}$ for all homogeneous objects $X \in \CC_g$ and $Y \in \CC_h$.

In the semisimple case there is a universal grading. Namely, assume that $\CC$ is an essentially small semisimple monoidal $K$-linear category with simple unit.
Then the \emph{chain group} $\Ch(\CC)$ of $\CC$ is the group generated by the formal symbols $[X]$ for the simple objects $X$ of $\CC$, subject to the relation $[X] [Y] = [Z]$ whenever $Z$ appears as a subobject of $X \otimes Y$.
Note that if $X$ is simple and there exists a right dual $X^\vee$ of $X$, then the class $[X^\vee]$ defines the inverse of $[X]$ in $\Ch(\CC)$. Of course the left duals have the same property.

The chain group defines a $\Ch(\CC)$-grading on $\CC$ in the obvious way: $\CC_g$ consists of direct sums of simple objects $X$ such that $[X]=g$. The following is immediate by definition.

\begin{lemma}\label{lem:ss-cat-chain-grading}
If $\CC$ is an essentially small semisimple monoidal $K$-linear category with simple unit, then to be given a full $\Gamma$-grading on $\CC$ is equivalent to be given a group homomorphism $\Ch(\CC)\to\Gamma$.
\end{lemma}

\smallskip

The chain group can also be described as follows. For a tensor functor $F\colon\CC\to\CC$ denote by $\Aut^\otimes(F)$ the group of natural monoidal automorphisms of $F$.  The following result is more or less known. For the representation category of a compact group $G$, it reduces to M\"uger's result~\cite{MR2130607} which states that $\Ch(\Rep G)$ is the Pontryagin dual of the center of $G$, while for the fusion categories it is proved in~\cite{MR2383894}.

\begin{proposition}\label{prop:chain-grp-and-nat-aut-id-fuct-duality}
Assume that $\CC$ is an essentially small semisimple monoidal $K$-linear category with simple unit, such that $\End_\CC(U) = K$ for any simple object $U$. Then there is a natural group isomorphism
$$
\Hom(\Ch(\CC), K^\times) \simeq \Aut^\otimes(\Id_\CC),
$$
where the group structure of $\Hom(\Ch(\CC), K^\times)$ is given by pointwise multiplication.
\end{proposition}

\bp
Let $\phi$ be a group homomorphism from $\Ch(\CC)$ to $K^\times$. We can define a natural transformation $\xi^\phi_X \colon X \to X$ by setting $\xi^\phi_X = \phi([X]) \iota_X$ for simple $X$ and extending it to general objects of $\CC$ by direct sum decomposition into simple objects. The group homomorphism property of $\phi$ implies that $\xi^\phi$ belongs to $\Aut^\otimes(\Id_\CC)$. It is straightforward to see that $\phi \mapsto \xi^\phi$ is a group homomorphism.

Reversing the above correspondence, starting from $\xi \in \Aut^\otimes(\Id_\CC)$, we can define a map $\phi^\xi \colon \Ch(\CC) \to K^\times$ by the characterization $\phi^\xi([X]) \iota_X = \xi_X$. As $\xi_{X \otimes Y} = \xi_X \otimes \xi_Y$, any subobject $Z$ of $X \otimes Y$ satisfies $\xi_Z = \phi^\xi([X]) \phi^\xi([Y]) \iota_Z$. Thus, $\phi^\xi$ is a group homomorphism. It is clear from the construction that $\xi \mapsto \phi^\xi$ and $\phi \mapsto \xi^\phi$ are inverse to each other.
\ep

Next, let us translate the above to the algebraic framework of representation category of counital coalgebras. Denote by $K\Gamma$ the group algebra of $\Gamma$ over $K$. For a coalgebra $A$ denote by $\Corep (A)$ the category of finite dimensional left comodules over~$A$. Recall that a homomorphism of coalgebras $f\colon A \to B$ is said to be \emph{cocentral} if $$f(a_{(1)}) \otimes a_{(2)} = f(a_{(2)}) \otimes a_{(1)}$$ holds for any $a \in A$. This condition is equivalent to saying that the image of the transpose map $f^t\colon B^*\to A^*$ lies in the center of the algebra $A^*$ of linear functionals on $A$, endowed with the convolution product.

\begin{lemma}\label{lem:gradings}
For any coalgebra $A$ there is a one-to-one correspondence between
\begin{enumerate}
\item full $\Gamma$-gradings on $\Corep(A)$;
\item $\Gamma$-gradings on $A$ such that $\Delta(A_g)\subset A_g\otimes A_g$ for all $g\in\Gamma$;
\item cocentral coalgebra homomorphisms $A\to K\Gamma$.
\end{enumerate}

Similarly, if $A$ is a Hopf algebra, then there is a one-to-one correspondence between
\begin{itemize}
\item[(i$'$)] full $\Gamma$-gradings on the monoidal category $\Corep(A)$;
\item[(ii$'$)] $\Gamma$-gradings on $A$ such that $\Delta(A_g)\subset A_g\otimes A_g$ and $A_g A_h\subset A_{g h}$ for all $g,h\in\Gamma$;
\item[(iii$'$)] cocentral Hopf algebra homomorphisms $A\to K\Gamma$.
\end{itemize}
\end{lemma}

\bp This is a routine verification, but let us present a brief argument for the reader's convenience.

Assume (i), so that we are given a full $\Gamma$-grading on $\CC=\Corep(A)$.
It is well-known that every element $a \in A$ generates a finite dimensional left subcomodule, namely, the linear span of $a \lhd \phi = \phi(a_{(1)}) a_{(2)}$ for $\phi\in A^*$. It follows that if we denote by $A_g$ the subspace of $A$ consisting of elements that generate subcomodules belonging to $\CC_g$, then $A=\oplus_{g\in\Gamma} A_g$.
 Moreover, for any $\phi \in A^*$, the map $a \mapsto \phi \rhd a = \phi(a_{(2)}) a_{(1)}$ is a homomorphism of left $A$-comodules. Since any $A$-comodule homomorphism of $A_g$ to $A_h$ is zero for $g \neq h$, we must have $\Delta(A_g)\subset A_g\otimes A_g$, and we obtain (ii).

Next, given a $\Gamma$-grading on $A$ as in (ii), we define $f\colon A\to K\Gamma$ by $f(a)=\eps(a)g$ for $a\in A_g$. Then $f$ is cocentral. Note that the grading on $A$ is recovered from $f$ by $A_g=\{a\in A\mid a_{(1)}\otimes f(a_{(2)})=a\otimes g\}$.

Finally, given a cocentral coalgebra homomorphism $f\colon A\to K\Gamma$, any comodule over $A$ can be considered as a comodule over $K\Gamma$, that is, as a $\Gamma$-graded vector space. The cocentrality assumption implies that this $\Gamma$-grading respects the $A$-comodule structure. Hence we obtain a $\Gamma$-grading on $\Corep(A)$.

\smallskip

The equivalence of (i$'$)--(iii$'$) can be argued in the same way. For example, from (ii$'$), when $a \in A_g$ and $b \in A_h$, one obtains $f(a) f(b) = \couni(a) \couni(b) g h= \couni(a b) g h = f(a b)$, hence $f$ is a bialgebra homomorphism, thus it is a Hopf algebra homomorphism.
\ep

Let us note that if any (or all) of the above conditions (i$'$)--(iii$'$) holds, the antipode~$S$ of~$A$ satisfies $S(A_g) = A_{g^{-1}}$ for each $g$, since it is the unique linear map satisfying $S(a_{(1)}) a_{(2)} = \couni(a) 1 = a_{(1)} S(a_{(2)})$ and we know that $1 \in A_e$. Moreover, Lemma~\ref{lem:ss-cat-chain-grading} states the following in this setup. If $\CC=\Corep(A)$ for a cosemisimple Hopf algebra $A$, we have a surjective cocentral Hopf algebra homomorphism $A\to K\Ch(\Corep(A))$, and any cocentral Hopf algebra homomorphism $A\to K\Gamma$ factors through it.

\subsection{Group actions on monoidal categories}
\label{sec:grp-act-mon-cat}

Given a group $\Gamma$, consider the monoidal category~$\underline{\Gamma}$ with objects the elements of $\Gamma$, no nontrivial morphisms, and with the tensor structure given by the product in $\Gamma$. Assume that $\CC$ is a monoidal $K$-linear category, and  denote by $\underline{\Aut}^\otimes(\CC)$ the monoidal category of monoidal autoequivalences of $\CC$, with the tensor structure given by the composition of monoidal functors. Then a \emph{(weak) action} of $\Gamma$ on $\CC$ is a tensor functor $\underline{\Gamma}\to \underline{\Aut}^\otimes(\CC)$, see, e.g.,~\cite{MR1815142}.

Spelling out the meaning of this definition, an action of $\Gamma$ on $\CC$ is defined by monoidal autoequivalences $\alpha^g$ of $\CC$ and natural monoidal isomorphisms  $\eta^{g,h}$ from $\alpha^g \alpha^h$ to $\alpha^{g h}$ such that
\begin{equation}
\label{eq:assoc-eta-cond}
\eta^{g, h k} \alpha^g(\eta^{h,k}) = \eta^{g h, k} \eta^{g,h}.
\end{equation}
This can be explicitly written as the equality of morphisms
$$
\eta^{g, h k}_X \alpha^g(\eta^{h,k}_X) = \eta^{g h, k}_X \eta^{g,h}_{\alpha^k (X)} \colon \alpha^g \alpha^h \alpha^k( X) \to \alpha^{g h k}(X).
$$
Note also that the monoidality of $\eta^{g,h}$ means commutativity of the following diagrams:
\begin{equation*}
\label{eq:eta-g-h-mon-transf}
\xymatrix@R=2em@C=-2em{
& \alpha^g \alpha^h(X \otimes Y) \ar[rr]^{\eta^{g, h}_{X\otimes Y}}  & & \alpha^{g h}(X \otimes Y) \\
\alpha^g(\alpha^h (X) \otimes \alpha^h (Y)) \ar[ur]^{\alpha^g(\alpha^h_2)} & & & &\alpha^{g h}(X) \otimes \alpha^{g h}(Y) \ar[ul]_{\alpha^{g h}_2},\\
& & \alpha^g \alpha^h(X) \otimes \alpha^g \alpha^h(Y) \ar[urr]_{\eta^{g,h}_X \otimes \eta^{g,h}_Y}\ar[ull]^{\alpha^g_2}
}
\end{equation*}
where $\alpha^g_2$ denotes the monoidal structure of $\alpha^g$.

We will denote an action by a pair $(\alpha,\eta)$ as above. Replacing the functor $\underline{\Gamma}\to \underline{\Aut}^\otimes(\CC)$ by a naturally monoidally isomorphic one, we may, and will, assume that  $\alpha^e = \Id_\CC$ and $\eta^{e, g}$ and $\eta^{g, e}$ are the identity isomorphisms. An action is called \emph{strict} if we have equality of monoidal autoequivalences $\alpha^g \alpha^h = \alpha^{g h}$ as well as the equalities $\eta^{g,h} = \iota$ for all $g$ and $h$. An example of  a strict action can be obtained by taking $\CC=\Corep(A)$ for a Hopf algebra $A$ and an action of $\Gamma$ on $A$ by Hopf algebra automorphisms.

\smallskip

By Tambara's result~\cite{MR1815142}, if $\Gamma$ acts on $\CC$, we can construct a new monoidal category $\CC \rtimes_{\alpha,\eta} \Gamma$ over $K$, analogous to the crossed product construction of Hopf algebras. The underlying linear category is the same as the Deligne product $\CC \boxtimes \CC_\Gamma$, where $\CC_\Gamma$ is the monoidal category over $K$ whose simple classes are given by the elements of $\Gamma$. In other words, $\CC_\Gamma=\Corep(K\Gamma)$ is the category of finite dimensional $\Gamma$-graded vector spaces. The monoidal structure of $\CC \boxtimes \CC_\Gamma$ is given by $(X \boxtimes g) \otimes (Y \boxtimes h) = (X \otimes \alpha^g(Y)) \boxtimes g h$ at the level of objects, and by $(T \boxtimes \iota_g) \otimes (S \boxtimes \iota_h) = (T \otimes \alpha^g(S)) \boxtimes \iota_{g h}$ for $T \in \CC(X, W)$ and $S \in \CC(Y, Z)$. The tensor unit is given by $\un \boxtimes e$, while the new associativity morphisms $\tilde{\Phi}$ are given by
\begin{equation}
\label{eq:crossed-prod-associator}
(\iota \otimes \alpha^g_2)(\iota\otimes(\iota\otimes \eta^{g,h}))^{-1}\Phi  \boxtimes \iota_{g h k}\colon (X \boxtimes g \otimes Y \boxtimes h) \otimes Z \boxtimes k \to X \boxtimes g \otimes (Y \boxtimes h \otimes Z \boxtimes k),
\end{equation}
where $\Phi$ denotes the associativity morphisms in $\CC$.

\smallskip

In what follows we will be particularly interested in actions $(\alpha,\eta)$ such that $\alpha^g$ is naturally monoidally isomorphic to the identity functor on $\CC$ for all $g\in\Gamma$. In this case choose such isomorphisms $\zeta^g\colon \alpha^g\to\Id_\CC$, with $\zeta^e=\iota$, and define natural monoidal automorphisms $\mu^{g,h}$ of~$\Id_\CC$~by
\begin{equation} \label{eq:cocycle}
\mu^{g,h}\zeta^{g h} \eta^{g,h} = \zeta^h \zeta^g.
\end{equation}

\begin{lemma}
The natural isomorphisms $\mu^{g,h}$ satisfy the cocycle identity
$$
\mu^{g,h}_X\mu^{gh,k}_X=\mu^{h,k}_X\mu^{g,hk}_X.
$$
\end{lemma}

\bp
This follows immediately by multiplying \eqref{eq:assoc-eta-cond} on the left by $\zeta^{ghk}$
and using the definition of $\mu$.
\ep

If $\CC$ is essentially small, so that $\Aut^\otimes(\Id_\CC)$ is a group, then this can be formulated by saying that $\mu=(\mu^{g,h})_{g,h}$ is an $\Aut^\otimes(\Id_\CC)$-valued $2$-cocycle on
$\Gamma$, where $\Aut^\otimes(\Id_\CC)$ is considered as a trivial $\Gamma$-module.

Our main example of this situation is as follows. Consider a Hopf algebra $A$.
We let $H^1(A)$ denote the group of algebra homomorphisms $A \rightarrow K$, with convolution product $(\mu * \nu)(a) = \mu(a_{(1)}) \nu(a_{(2)})$. Similarly, $H_\ell^1(A)$ denotes the subgroup of $H^1(A)$ formed by strongly central (or lazy) elements, i.e.,~those $\omega \in H^1(A)$ such that $\omega * \iota=\iota * \omega$. If $A$ is finite dimensional, so that $A^*$ is again a Hopf algebra, then $H^1(A)$ is the group of group-like elements in $A^*$ and $H_\ell^1(A)$ is the subgroup of group-like elements that are central in $A^*$.

Suppose that we are given a group homomorphism $\Gamma \rightarrow H^1(A)/ H^1_\ell(A)$. Lifting it to a map $\Gamma\to H^1(A)$ we get a pair $(\phi, \mu)$, where $\phi$ is a map
$\Gamma \rightarrow H^1(A)$, $g \mapsto \phi_g$, and   $\mu$ is a $2$-cocycle on~$\Gamma$ with values in the trivial $\Gamma$-module $H^1_\ell(A)$, such that
$$
\phi_e = \varepsilon, \quad \phi_g * \phi_h =\phi_{gh}* \mu(g,h)=\mu(g,h)*\phi_{gh}.
$$
Note that for $g \in \Gamma$, we have $\phi_g^{-1} = \phi_{g^{-1}} * \mu(g,g^{-1})^{-1}$, and $\mu(g,g^{-1})=\mu(g^{-1},g)$. Define an action of $\Gamma$ on $A$ by the Hopf algebra automorphisms
$$
\alpha_g(a) = \phi_g \lact a \ract \phi_g^{-1} = \phi_g^{-1}(a_{(1)}) a_{(2)} \phi_g(a_{(3)})=
\mu(g,g^{-1})(S(a_{(1)}))\phi_{g^{-1}}(a_{(2)}) a_{(3)} \phi_{g}(a_{(4)}),
$$
which depends only on the initial homomorphism $\Gamma \rightarrow H^1(A)/ H^1_\ell(A)$. We call such actions \emph{almost adjoint}.

We then get a strict action of $\Gamma$ on $\CC=\Corep(A)$, which we continue to denote by $\alpha$: given a corepresentation $\delta\colon V\to A\otimes V$, we let $\alpha^g(V,\delta)=(V,(\alpha_g\otimes\iota)\delta)$. The autoequivalences $\alpha^g$ are isomorphic to $\Id_\CC$. Namely, for every $X=(V,\delta)\in\CC$ consider the representation $\pi_X$ of the algebra $A^*$ on $V$ defined by $\pi_X(\omega)=(\omega\otimes\iota)\delta$. Then the morphisms $\zeta^g_X=\pi_X(\phi_g)\colon \alpha^g(X)\to X$ define a natural monoidal isomorphism $\zeta^g\colon \alpha^g\to\Id_\CC$. The corresponding cocycle $(\mu^{g,h})_{g,h}$ defined by \eqref{eq:cocycle} is given by
\begin{equation} \label{eq:cocycle2}
\mu^{g,h}_X=\pi_X(\mu(g,h)).
\end{equation}

\bigskip

\section{Graded twisting}\label{sec:grtw}

\subsection{Categorical formulation}

Let $(\alpha, \eta)$ be an action of $\Gamma$ on a $\Gamma$-graded monoidal $K$-linear category $\CC$.
We say that $(\alpha,\eta)$ is an \emph{invariant} action if each $\alpha^g$ preserves the homogeneous subcategories $\CC_h$ for all $h\in \Gamma$.

\begin{definition}
Given an invariant action $(\alpha, \eta)$ of $\Gamma$ on $\CC$, we denote by $\CC^{t,(\alpha, \eta)}$ the full monoidal $K$-linear subcategory of $\CC \rtimes_{\alpha,\eta} \Gamma$ obtained by taking direct sums of the objects $X \boxtimes g$ for $g \in \Gamma$ and $X \in \CC_g$, and call $\CC^{t,(\alpha, \eta)}$ the \emph{graded twisting of $\CC$ by the action of $\Gamma$}. For strict actions we write~$\CC^{t,\alpha}$.
\end{definition}

By construction, $\CC^{t,(\alpha, \eta)}$ is equivalent to $\CC$ as a $\Gamma$-graded $K$-linear category. Identifying $\CC$ and $\CC^{t,(\alpha,\eta)}$ as $K$-linear categories, we may express the twisted monoidal structure as an operation on~$\CC$ given by $X \otimes_{\alpha,\eta} Y = X \otimes \alpha^g(Y)$ if $X \in \CC_g$. However, in some cases instead of changing the monoidal structure we can change the associativity morphisms in $\CC$. Namely, we have the following result.

\begin{theorem} \label{thm:new-associator}
Assume an invariant action $(\alpha,\eta)$ of $\Gamma$ on $\CC$ is such that $\alpha^g\simeq\Id_\CC$ for all $g\in\Gamma$. Choose natural monoidal isomorphisms $\zeta^g\colon\alpha^g\to\Id_\CC$ and define a cocycle $\mu=(\mu^{g,h})_{g,h}$ by~\eqref{eq:cocycle}. Then
\begin{enumerate}
\item there are new associativity morphisms $\Phi^{\mu}$ on $\CC$ such that
$$
\Phi^{\mu} = (\iota\otimes(\iota \otimes \mu^{g,h}_Z))\Phi \colon (X \otimes Y) \otimes Z \to X \otimes (Y \otimes Z)
$$
for $X \in \CC_g$ and $Y \in \CC_h$;
\item there is a monoidal equivalence $F\colon\CC^{t,(\alpha,\eta)}\to (\CC,\Phi^\mu)$ defined by $F(X \boxtimes g) =X$ for $X\in\CC_g$, with the tensor structure
$$
F_2 \colon F(X \boxtimes g) \otimes F(Y \boxtimes h) \to F(X \boxtimes g \otimes Y \boxtimes h)
$$
represented by $(\iota \otimes \zeta^g)^{-1}\colon X \otimes Y  \to  X \otimes \alpha_g (Y)$ for $X \in \CC_g$ and $Y \in \CC_h$.
\end{enumerate}
\end{theorem}

\bp
Part (i) follows easily from the cocycle property of $\mu$, but it will also follow from the proof of~(ii) (cf.~\cite{MR1047964}). For part (ii), in turn, it suffices to check that $(F,F_2)$ transforms the associator on~$\CC^{t,(\alpha,\eta)}$ into $\Phi^\mu$. Recalling formula \eqref{eq:crossed-prod-associator} for the associator on $\CC\rtimes_{\alpha,\eta}\Gamma$, this amounts to verifying the equality
\begin{multline}\label{eq:associator1}
(\iota\otimes\alpha^g_2)(\iota\otimes(\iota\otimes\eta^{g,h}))^{-1}\Phi(\iota\otimes\zeta^{gh})^{-1}((\iota\otimes\zeta^g)^{-1}\otimes\iota)\\
=(\iota\otimes\zeta^g)^{-1}(\iota\otimes(\iota\otimes\zeta^h)^{-1})(\iota\otimes(\iota \otimes \mu^{g,h}_Z))\Phi
\end{multline}
of morphisms $(X\otimes Y)\otimes Z\to X\otimes\alpha^g(Y\otimes\alpha^h(Z))$ for $X\in\CC_g$, $Y\in\CC_h$, $Z\in\CC_k$.

The assumption that $\zeta^g$ is a natural isomorphism of monoidal functors implies that $(\zeta^g)^{-1}= \alpha^g_2(\zeta^g\otimes\zeta^g)^{-1}$. Using also that $\zeta^h\zeta^g=\mu^{g,h}\zeta^{gh}\eta^{g,h}$, we see that the right hand side of \eqref{eq:associator1} equals
$$
(\iota\otimes\alpha^g_2)(\iota\otimes(\iota\otimes\eta^{g,h}))^{-1}(\iota\otimes(\zeta^g\otimes \zeta^{gh}))^{-1}\Phi.
$$
But this equals the left hand side of \eqref{eq:associator1} by naturality of $\Phi$.
\ep

In the semisimple case the new associativity morphisms can be expressed in terms of usual $3$-cocycles as follows.

\begin{proposition} \label{prop:3-cocycle}
Let $\CC$ be as in the previous theorem. Moreover, assume that $\CC$ is essentially small, semisimple, with simple unit, such that $\End_\CC(U) = K$ for any simple object $U$. Let $q\colon \Ch(\CC)\to\Gamma$ denote the homomorphism corresponding to the $\Gamma$-grading on $\CC$. Then there is a well-defined cocycle $c\in Z^3(\Ch(\CC);K^\times)$ such that
$$
c(g,h,k) \iota_Z =\mu^{q(g),q(h)}_Z
$$
for any $g,h,k\in\Ch(\CC)$ and any simple object $Z\in\CC$ with $[Z]=k$.
\end{proposition}

\bp By Proposition~\ref{prop:chain-grp-and-nat-aut-id-fuct-duality}, for any $s,t\in\Gamma$ we can define a homomorphism $\Ch(\CC)\to K^\times$ by $[Z]\mapsto\mu^{s,t}_Z$. This shows that the map $c\colon\Ch(\CC)^3\to K^\times$ in the formulation is well-defined. The cocycle identity for $c$ can be checked directly, but it also follows from the pentagon relation for~$\Phi^\mu$, as
$$
\Phi^\mu=c(g,h,k)\Phi\colon (X\otimes Y)\otimes Z\to X\otimes (Y\otimes Z)
$$
for $[X]=g$, $[Y]=h$ and $[Z]=k$.
\ep

\begin{remark}
The cocycle $c$ in the previous proposition does not necessarily descend to $\Gamma$. Nevertheless, if $\Gamma$ is abelian, then $c$ at least descends to the abelianization of $\Ch(\CC)$, since the homomorphism $\Ch(\CC)\to K^\times$, $[Z]\mapsto\mu^{s,t}_Z$, in the above proof factors through $\Ch(\CC)^{\abel}$.
\end{remark}

\begin{remark}
Assume $(\alpha,\eta)$ is an action of $\Gamma$ on a $\Gamma$-graded monoidal $K$-linear category $\CC$.
Let us say that the action is \emph{equivariant} if $\alpha^g(\CC_h) = \CC_{g h g^{-1}}$ holds for any $g, h \in \Gamma$. If $\alpha$ is strict, then~$\CC$ is called a crossed $\Gamma$-category over $K$~\cite{MR2674592}*{Section~VI}. For any equivariant action we have $X \otimes \alpha^{g^{-1}}(Y) \in \CC_{g g^{-1} h g} = \CC_{h g}$, which implies that $X \boxtimes g^{-1} \otimes Y \boxtimes h^{-1}$ belongs to $\CC_{h g} \boxtimes (h g)^{-1}$. Thus, the objects of the form $X \boxtimes g^{-1}$ for $X \in \CC_g$ span a monoidal subcategory $\CC^{a,(\alpha,\eta)}$, which is equivalent to $\CC$ as a $\Gamma$-graded $K$-linear category. We note that if $\Gamma$ is abelian, then the notions of invariant and equivariant actions are the same, but the monoidal categories $\CC^{a,(\alpha,\eta)}$ and $\CC^{t,(\alpha,\eta)}$ are not equivalent in general, see Section~\ref{sec:abelian-groups}.
\end{remark}

\subsection{Hopf algebraic formulation}
\label{ssec:hopf-alg-formul}

Let $A$ be a Hopf algebra. Assume we are given a $\Gamma$-grading on $A$ such that $\Delta(A_g)\subset A_g\otimes A_g$ and $A_gA_h\subset A_{gh}$ for all $g,h\in\Gamma$. An action $\alpha$ of $\Gamma$ on $A$ by Hopf algebra automorphisms is called \emph{invariant} if $\alpha_g(A_h)=A_h$ for all $g,h\in\Gamma$.

By Lemma~\ref{lem:gradings} these assumptions can be formulated by saying that we are given a cocentral Hopf algebra homomorphism $p\colon A\to K\Gamma$ and an action $\alpha$ of $\Gamma$ on $A$ such that $p\alpha_g=p$ for all $g\in\Gamma$. We then say that the pair $(p,\alpha)$ is an \emph{invariant cocentral action} of $\Gamma$ on $A$. Any such action defines a strict invariant action of $\Gamma$ on $\Corep(A)$ and therefore we get a twisted category $\Corep(A)^{t,\alpha}$. We want to define a Hopf algebra of which $\Corep(A)^{t,\alpha}$ is the corepresentation category.

Consider the crossed product $A \rtimes_\alpha \Gamma$. We realize it as the tensor product $A \otimes K\Gamma$ with the twisted product $(a \otimes g) (b \otimes h) = a \alpha_g(b) \otimes g h$. It is a Hopf algebra with the tensor product of original coproducts $\Delta(a \otimes g)= a_{(1)} \otimes g \otimes a_{(2)} \otimes g$, the counit $\couni(a \otimes g) = \couni(a)$, and the antipode $S(a \otimes g) = S(\alpha_{g^{-1}} (a)) \otimes g^{-1}$.

From this presentation it is straightforward that we have a monoidal equivalence $$\Corep(A \rtimes_\alpha \Gamma)\simeq\Corep(A) \rtimes_{\alpha} \Gamma.$$
It is easy to see that the full monoidal subcategory $\Corep(A)^{t,\alpha}\subset \Corep(A) \rtimes_{\alpha} \Gamma$ is rigid. It follows that the matrix coefficients of corepresentations of $A \rtimes_\alpha \Gamma$ lying in $\Corep(A)^{t,\alpha}$ span a Hopf subalgebra of $A \rtimes_\alpha \Gamma$ with corepresentation category $\Corep(A)^{t,\alpha}$. We thus arrive at the following definition.

\begin{definition}
Given an invariant cocentral action $(p,\alpha)$ of $\Gamma$ on a Hopf algebra $A$, \emph{the graded twisting of $A$ by the action of $\Gamma$} is the Hopf subalgebra $A^{t,(p, \alpha)}$ of $A \rtimes_\alpha \Gamma$ spanned by the elements of the form $a\otimes g$ with $a\in A_g=\{a\in A\mid a_{(1)}\otimes p( a_{(2)})=a\otimes g\}$ and $g\in\Gamma$.
If $p$ is unambiguous from the context, we will also denote this Hopf algebra by $A^{t,\alpha}$.
\end{definition}

By construction we have $\Corep(A^{t,\alpha})\simeq\Corep(A)^{t,\alpha}$, but of course no categorical considerations are needed just to define $A^{t,\alpha}$. Observe that by working with homogeneous components we immediately see that $A^{t,(p, \alpha)}\subset A \rtimes_\alpha \Gamma$ is a sub-bialgebra, while the fact that it is closed under the antipode follows from $S(A_g)\subset A_{g^{-1}}$.

The following description of $A^{t,\alpha}$ is very useful.

\begin{proposition}\label{prop:construction}
The crossed product $A \rtimes_\alpha \Gamma$ becomes a $K\Gamma$-comodule by the coaction
\begin{equation}\label{eq:coact-K-Gamma-on-inv-cocentr-act}
 A \rtimes_\alpha \Gamma \rightarrow  (A \rtimes_\alpha \Gamma) \otimes K\Gamma,\quad
 a \otimes g \mapsto a_{(1)} \otimes g \otimes p(a_{(2)})g^{-1},
\end{equation}
the space $(A \rtimes_\alpha \Gamma)^{\co K \Gamma}$ of invariant elements (that is, elements $x$ satisfying $x \mapsto x \otimes 1$ for the above coaction) coincides with $A^{t,\alpha}$, and the map
\begin{equation*}
j\colon A \rightarrow  A^{t,\alpha}=(A \rtimes_\alpha \Gamma)^{\co K \Gamma}, \quad
a \mapsto a_{(1)} \otimes p(a_{(2)})
\end{equation*}
is a coalgebra isomorphism. Moreover, the map $\tilde{p} = \varepsilon \otimes \iota \colon a \otimes g \mapsto \varepsilon(a) g$ restricts to a cocentral Hopf algebra homomorphism $A^{t,\alpha}\to K \Gamma$ satisfying $\tilde{p}j=p$.
\end{proposition}

\begin{proof}
This is easily verified by working with homogeneous components of $A$.
\end{proof}

We now describe the graded twisting operation in terms of \emph{pseudo-$2$-cocycles}, when the action is almost adjoint. Recall that a linear map $\sigma \in (A \otimes A)^*$ is said to be a pseudo-$2$-cocycle (see \cite{MR1395206} for an equivalent formulation) if it is convolution invertible, the product
$$
A \times A \ni (a,b) \mapsto \sigma(a_{(1)},b_{(1)}) \sigma^{-1}(a_{(3)},b_{(3)}) a_{(2)}b_{(2)} \in A
$$
is associative with the initial $1$ still being the unit, and  the bialgebra $A^\sigma$ equipped with the initial coproduct and the new product is a Hopf algebra.
The pseudo-$2$-cocycle $\sigma$ defines then a quasi-tensor equivalence (tensor equivalence ``minus compatibility with the associators'') between $\Corep(A)$ and $\Corep(A ^\sigma)$, which is a tensor equivalence if and only if $\sigma$ is a $2$-cocycle in the usual sense~\cite{MR1213985}, see, e.g.,~\citelist{\cite{MR1047964}\cite{MR1408508}}.

Assume that we have an almost adjoint action $(p, \alpha)$ on $A$ as at the end of Section~\ref{sec:grp-act-mon-cat}. Thus~$\alpha$~is defined by a pair $(\phi,\mu)$ consisting of a map $\phi\colon \Gamma\to H^1(A)$ and a $2$-cocycle $\mu\colon\Gamma^2\to H^1_\ell(A)$. Recall that $\mu$ can be considered as a cocycle with values in $\Aut^\otimes(\Id_{\Corep(A)})$ by \eqref{eq:cocycle2}. Consider the associativity morphisms $\Phi^\mu$ on $\Corep(A)$ defined by $\mu$ as described in Theorem~\ref{thm:new-associator}. By that theorem we have a monoidal equivalence $\Corep(A^{t,\alpha})\simeq(\Corep(A),\Phi^\mu)$. From this get a fiber functor $\tilde F\colon(\Corep(A),\Phi^\mu)\to\Vect_K$ which is the forgetful functor on morphisms and objects, while its tensor structure $\tilde F_2\colon \tilde F(X)\otimes \tilde F(Y)\to \tilde F(X\otimes Y)$ is given by $\iota_X\otimes\pi_Y(\phi_g)$ for $X\in\Corep(A)_g$. It follows that $\tilde F_2^{-1}$ defines a pseudo-$2$-cocycle on $A$ such that the twisting of~$A$ by this pseudo-cocycle gives~$A^{t,\alpha}$. Let us formulate this more precisely.

Extend the maps $\phi_g$ and $\phi^{-1}_g$ on $\Gamma$ to $K\Gamma$ by linearity, so that $\phi^{-1}_x = \sum_g x_g \phi^{-1}_g$ for $x = \sum_g x_g g$ in $K\Gamma$. Similarly, extend $\mu$ by bilinearity to a map $K\Gamma\otimes K\Gamma\to H^1_\ell(A)$ . Define
\begin{equation} \label{eq:pseudo-cocycle}
 \sigma\colon A \otimes A \rightarrow K, \quad a \otimes b \mapsto \phi^{-1}_{p(a)}(b)  =\mu(p(a_{(1)}),S(p(a_{(2)})))(S(b_{(1)}))\phi_{S(p(a_{(3)}))}(b_{(2)}).
\end{equation}

\begin{theorem} \label{thm:pseudo-2-cocycle}
Assume an almost adjoint invariant cocentral action $(p,\alpha)$ of $\Gamma$ on $A$ is defined by a pair $(\phi,\mu)$ as above. Then
\begin{enumerate}
\item the element $\sigma\in(A\otimes A)^*$ defined by \eqref{eq:pseudo-cocycle} is a pseudo-$2$-cocycle on $A$, with the convolution inverse given~by
\begin{equation*}
 \sigma^{-1}\colon A \otimes A \rightarrow K, \quad a \otimes b \mapsto \phi_{p(a)}(b);
\end{equation*}
\item the map $j\colon A\to A \rtimes_\alpha \Gamma$ defines a Hopf algebra isomorphism of the twist $A^\sigma$ of $A$ by $\sigma$ onto $A^{t,\alpha}\subset A \rtimes_\alpha \Gamma$.
\end{enumerate}
\end{theorem}

\bp The theorem follows from the preceding discussion, but let us give an explicit argument.

\smallskip

Denote by $\sigma'$ the map defined in (i). Recall that if $a \in A_g$, we have $p(a) = \couni(a) g$. For any $b \in A$, we thus have
$$
(\sigma * \sigma') (a \otimes b) = \phi^{-1}_{p(a_{(1)})}(b_{(1)}) \phi_{p(a_{(2)})}(b_{(2)}) = \couni(a) \phi^{-1}_g(b_{(1)}) \phi_g(b_{(2)}) = \couni(a) \couni(b).
$$
By linearity in $a$, we obtain that $\sigma * \sigma'$ is equal to $\couni \otimes \couni$. By a similar argument we also obtain $\sigma'*\sigma = \varepsilon \otimes \varepsilon$, hence $\sigma' = \sigma^{-1}$ as claimed.

\smallskip

Next, we claim that for any elements $a$, $b$, and $c$ of $A$, we have
\begin{equation} \label{eq:associator2}
\sigma^{-1}(a_{(1)},b_{(1)} c_{(1)}) \sigma^{-1}(b_{(2)},c_{(2)})\sigma(a_{(2)},b_{(3)}) \sigma(a_{(3)} b_{(4)},c_{(3)})=
\mu(p(a),p(b))(c).
\end{equation}
First notice that $\phi_x(b c) = \phi_{x_{(1)}}(b) \phi_{x_{(2)}}(c)$ holds for $x \in K\Gamma$, which is by definition true if $x$ is actually in $\Gamma$, and in general follows from that by linearity. This implies
\begin{multline*}
\sigma^{-1}(a,b_{(1)} c_{(1)}) \sigma^{-1}(b_{(2)},c_{(2)}) = \phi_{p(a)}(b_{(1)} c_{(1)}) \phi_{p(b_{(2)})}(c_{(2)}) \\= \phi_{p(a_{(1)})}(b_{(1)})\phi_{p(a_{(2)})}(c_{(1)})\phi_{p(b_{(2)})}(c_{(2)}).
\end{multline*}
Next, the cocycle property of $\mu$ can be expressed as $\phi_x * \phi_y = \phi_{x_{(1)} y_{(1)}} * \mu(x_{(2)}, y_{(2)})$, which is again easy to see on group elements. Thus, the above expression is equal to
$$
\phi_{p(a_{(1)})}(b_{(1)})\phi_{p(a_{(2)} b_{(2)})}(c_{(1)})\mu(p(a_{(3)}), p(b_{(3)}))(c_{(2)}),
$$
which by the cocentrality of $p$ is equal to
\begin{multline*}
 \mu(p(a_{(1)}), p(b_{(1)}))(c_{(1)}) \phi_{p(a_{(3)}b_{(2)})}(c_{(2)}) \phi_{p(a_{(2)})}(b_{(3)}) \\
 = \mu(p(a_{(1)}), p(b_{(1)}))(c_{(1)}) \sigma^{-1}(a_{(2)}b_{(2)},c_{(2)}) \sigma^{-1}(a_{(3)},b_{(3)}).
\end{multline*}
This clearly implies \eqref{eq:associator2}.

\smallskip

Recall that the associativity condition for the twisted product in the definition of a  pseudo-$2$-cocycle means that, although the right hand side of \eqref{eq:associator2} is not $\eps(abc)$, it defines a central $3$-form, that is,
$$
\mu(p(a_{(1)}),p(b_{(1)}))(c_{(1)})a_{(2)}b_{(2)}c_{(2)}=\mu(p(a_{(2)}),p(b_{(2)}))(c_{(2)})a_{(1)}b_{(1)}c_{(1)}.
$$
But this is an immediate consequence of the cocentrality of $p$ and the centrality of $\mu(g,h)$.

We thus get a bialgebra $A^\sigma$ with product given by
$$
a \cdot b = \sigma(a_{(1)}, b_{(1)}) a_{(2)} b_{(2)} \sigma^{-1}(a_{(3)}, b_{(3)}).
$$
If $a_g \in A_g$, the right hand side becomes
$$
\phi_g^{-1}(b_{(1)}) a_g b_{(2)} \phi_g(b_{(3)}) = a_g \alpha_g(b).
$$
This shows that $A_g \ni a_g \mapsto a_g \otimes g$ is indeed an algebra homomorphism from $A^\sigma$ to $A \rtimes_\alpha \Gamma$, and induces a bialgebra isomorphism $A^\sigma \simeq A^{t,\alpha}$. Hence $A^\sigma$ is a Hopf algebra as well, and this proves (i) and (ii).
\ep

\begin{remark}
 It follows from Equation \ref{eq:associator2} that $\sigma$ is a $2$-cocycle if and only if $\mu$ is the trivial map, i.e.,~if and only if $\phi : \Gamma \rightarrow H^1(A)$ is a group morphism.
\end{remark}

We end the section by discussing the Grothendieck ring of a twisted graded Hopf algebra.
Denote by $\K(A)$ the Grothendieck ring of the tensor category $\Corep(A)$. By the Jordan--H\"{o}lder theorem, the isomorphism classes of simple objects form a basis of the underlying group, and in the cosemisimple case it can be identified with the subring of $A$ generated by the characters of the simple corepresentations. An invariant cocentral action $(p, \alpha)$ on $A$ defines an action of $\Gamma$ on $\K(A)$, which is trivial when the action is almost adjoint, and in this case we have $\K(A) \simeq \K(A^{t,\alpha})$ as rings, which follows directly either from Proposition \ref{prop:construction} or Theorem \ref{thm:pseudo-2-cocycle}. When the invariant cocentral action is not almost adjoint, the coalgebras $A$ and $A^{t, \alpha}$ are still isomorphic, hence their Grothendieck groups are isomorphic. At the ring level, we have the following result, which follows directly from $\Corep(A)^{t,\alpha} = \Corep(A^{t,\alpha})$ as well.

\begin{proposition}\label{prop:grothendieck}
 Given an invariant cocentral action $(p,\alpha)$ of $\Gamma$ on $A$, we have a ring isomorphism
$$\K(A^{t,\alpha}) \simeq (\K(A) \rtimes \Gamma)^{{\rm co}K \Gamma},$$
where $\Gamma$ acts on $\K(A)$ by the functoriality of the functor $\K$.
\end{proposition}

\subsection{Invariant cocentral actions by abelian groups}\label{sec:abelian-groups}

There is a slight asymmetry in our construction. If we start with an invariant cocentral action of $\Gamma$ on a Hopf algebra, then we get a new twisted Hopf algebra, but an action of $\Gamma$ is no longer defined on it. This is in itself not surprising, as typically the symmetries of deformations should also deform. The situation is better for abelian groups.

Thus, in this section we specialize to the case when $\Gamma$ is abelian to have an extra symmetry in our construction.

\begin{definition}
We denote by $\DD(\Gamma)$ the category whose objects are the triples $(A,p,\alpha)$, where $(p, \alpha)$ is an invariant cocentral action of $\Gamma$ on a Hopf algebra $A$, and whose morphisms $(A, p, \alpha)\rightarrow  (B, q, \beta)$ are the Hopf algebra homomorphisms $f\colon A \rightarrow B$ such that $q f=p$ and $\beta_g f=\alpha_g$ for any $g \in \Gamma$.
\end{definition}

In other words, the morphisms in $\DD(\Gamma)$ are the $\Gamma$-equivariant homomorphisms of $\Gamma$-graded Hopf algebras.

\begin{lemma}
If $\Gamma$ is abelian and $(A, p,\alpha)$ is an object of $\DD(\Gamma)$, then $\tilde{\alpha}_g(j(a)) = j(\alpha_{g}(a))$ for $g \in \Gamma$ defines an object $(A^{t,\alpha}, \tilde{p}, \tilde{\alpha})$ of $\mathcal D(\Gamma)$, where $j$ is the map defined in Proposition~\ref{prop:construction}.
\end{lemma}

\bp
We have a well-defined action of $\Gamma$ on $A\rtimes_\alpha\Gamma$ by the Hopf algebra automorphisms
$$
\alpha_g(a\otimes h)=\alpha_g(a)\otimes h.
$$
By restriction we get the required action of $\Gamma$ on $A^{t,\alpha}$.
\ep

Given objects $(A, p, \alpha)$ and $(B,q,\beta)$ of $\mathcal D(\Gamma)$ and a morphism $f \colon (A, p, \alpha)\rightarrow(B,q,\beta)$, the natural extension $f \otimes \iota$ is a Hopf algebra homomorphism between the crossed products by $\Gamma$, and by restriction we obtain a Hopf algebra homomorphism $A^{t,\alpha} \rightarrow B^{t,\beta}$. Moreover if $\Gamma$ is abelian, it is a morphism from $(A^{t,\alpha}, \tilde{p}, \tilde{\alpha})$ to $(B^{t,\beta}, \tilde{q}, \tilde{\beta})$ in the category $\mathcal D(\Gamma)$. Therefore our graded twisting procedure defines an endofunctor of the category $\mathcal D(\Gamma)$, which is, as we shall see soon, an autoequivalence.

Since $\Gamma$ is abelian, the set of its endomorphisms $\End(\Gamma)$ has a ring structure, with the product given by composition $(\phi \cdot \psi)(g) = \phi(\psi(g))$ and the sum given by pointwise product $(\phi + \psi)(g) = \phi(g) \psi(g)$. Using this structure, one obtains the `$a x + b$' semigroup structure on $\End(\Gamma)^2$, whose product is given by
$$
(\phi_1, \psi_1) \cdot (\phi_2, \psi_2) = (\phi_1 + \psi_1 \cdot \phi_2, \psi_1 \cdot \psi_2).
$$
Let us show next that there is a right action of this semigroup on the category $\DD(\Gamma)$ such that the twisting functor $(A, p,\alpha)\mapsto(A^{t,\alpha}, \tilde{p}, \tilde{\alpha})$ corresponds to the action by the element $(\iota,\iota)$ in this semigroup.

If $\phi \in \End(\Gamma)$ and $(A, p, \alpha) \in \DD(\Gamma)$, the composition $\alpha \phi$ defines again an invariant cocentral action of $\Gamma$. Moreover, for each $g\in \Gamma$, the maps $\tilde{\alpha}_g = \alpha^g \otimes \iota$ on $A \rtimes_{\alpha \phi} \Gamma$ define an action on~$A^{t,\alpha\phi}$. Thus we can define an endofunctor $F^a_\phi$ on $\DD(\Gamma)$ by setting $F^a_\phi(A, p, \alpha) = (A^{t,\alpha \phi}, \tilde{p}, \tilde{\alpha})$. We need to verify the additivity in $\phi$.

\begin{lemma}
Let $\phi, \psi$ be endomorphisms of $\Gamma$, and $(A, p, \alpha)$ be an object of $\mathcal D(\Gamma)$. Then $A^{t,\alpha(\phi + \psi)}$ is naturally isomorphic to $(A^{t,\alpha\phi})^{t,\tilde{\alpha}\psi}$, given by the map $a_g \otimes g \mapsto a_g \otimes g \otimes g$ for $a_g \in A_g$.
\end{lemma}

\begin{proof}
The map in the assertion is induced by the $\Gamma$-graded coalgebra isomorphisms $j\colon A \to A^{t,\alpha(\phi + \psi)}$ and $j^2 \colon A \to (A^{t,\alpha\phi})^{t,\tilde{\alpha}\psi}$. Hence we only need to verify that it is compatible with the algebra structures.

If $a_g \in A_g$ and $b_h \in A_h$, one has
$$
((a_g \otimes g) \otimes g) ((b_h \otimes h) \otimes h) = ((a_g \otimes g) \tilde{\alpha}_{\psi(g)} (b_h \otimes h)) \otimes g h = ((a_g \otimes \alpha_{\phi(g)} \alpha_{\psi(g)}(b_h)) \otimes h)\otimes h.
$$
Since $\alpha_{\phi(g)} \alpha_{\psi(g)} = \alpha_{\phi(g) \psi(g)}$, one sees the compatibility with the algebra structures.
\end{proof}

It follows from the lemma that the graded twisting functor is indeed an autoequivalence of~$\mathcal D(\Gamma)$, with inverse given by the functor $(A,p,\alpha) \mapsto (A^{t,\alpha (-\iota)}, \tilde{p}, \tilde{\alpha})$, where $-\iota$ is the automorphism of $\Gamma$ given by $-\iota(g) = g^{-1}$.

As an immediate and useful consequence, we have the following result.

\begin{lemma}\label{iso-graded-twisted}
 Let $(A,p,\alpha)$ and $(B,q,\beta)$ be objects in $\mathcal D(\Gamma)$. Assume that we have a Hopf algebra map $$f : B \rightarrow A^{t,\alpha}$$ inducing a morphism $(B,q,\beta) \rightarrow (A^{t,\alpha},\tilde{p},\tilde{\alpha})$ in $\mathcal D(\Gamma)$. Then $f$ is an isomorphism if and only if there exists a Hopf algebra map $$f' : A \rightarrow B^{t,\beta(-\iota)}$$ inducing a morphism $(A,p,\alpha) \rightarrow (B^{t,\beta(-\iota)},\tilde{q},\tilde{\beta})$ such that $(f'\otimes \iota)f$ and $(f\otimes \iota)f'$ are the canonical isomorphisms $B\simeq B^{t,0}$ and $A\simeq A^{t,0}$, respectively.
\end{lemma}

Next, define $F^m_\psi(A, p, \alpha) = (A, p, \alpha \psi)$.

\begin{lemma}
 The endofunctors $F^a_\phi$ and $F^m_\psi$ satisfy $F^a_\phi F^m_\psi = F^m_\psi F^a_{\psi \cdot \phi}$.
\end{lemma}

\bp
Expanding the definitions, we see that both $F^a_\phi F^m_\psi$ and $F^m_\psi F^a_{\psi \cdot \phi}$ send the object $(A, p, \alpha)$ to $(A^{t,\alpha \phi}, \tilde{p}, \tilde{\alpha} \psi)$.
\ep

By this lemma, we obtain a right action of the $(a x + b)$-semigroup of $\End(\Gamma)$ on $\DD(\Gamma)$, defined~by
$$
(A, p, \alpha) \cdot (\phi, \psi) = F^m_\psi F^a_\phi (A, p, \alpha)=(A^{t,\alpha \phi}, \tilde{p}, \tilde{\alpha}\psi).
$$
As an example, with the automorphism $-\iota(g) = g^{-1}$, the action of $(-\iota,-\iota)$ defines an involution on $\DD(\Gamma)$.

\subsection{Examples} To illustrate the graded twisting procedure, we end the section with examples.

\begin{example}\label{ex-sl2}
First let us explain the relation of quantum $\SL(2)$ for the parameters  $q$ and $-q$ (see~\cite{MR3340190} for the compact case). At the categorical level this can be explained also from the universality of the Temperley--Lieb category, see~\cite{MR2046203}*{Remark~2.2}.

Denote by $a,b,c,d$ the standard generators of $\mathcal O(\SL_q(2))$. There is a cocentral Hopf algebra map
\begin{equation*}
   p: \mathcal O(\SL_q(2))  \rightarrow K\mathbb Z_2,\quad
\begin{pmatrix} a & b \\ c & d  \end{pmatrix} \mapsto \begin{pmatrix} g & 0 \\ 0 & g\end{pmatrix}
  \end{equation*}
where $g$ denotes the generator of $\mathbb Z_2$.
There is also an involutive Hopf algebra automorphism $\alpha=\alpha_g$ of $\mathcal O(\SL_q(2))$ defined by
$$
\begin{pmatrix} a & b \\ c & d  \end{pmatrix} \mapsto \begin{pmatrix} a & -b \\ -c & d  \end{pmatrix} = \begin{pmatrix} i & 0 \\0 & -i\end{pmatrix} \begin{pmatrix} a & b \\ c & d  \end{pmatrix} \begin{pmatrix} -i & 0 \\0 & i\end{pmatrix}.
$$
It is straightforward to check that there exits a Hopf algebra map
\begin{equation*}
 \mathcal O(\SL_{-q}(2)) \rightarrow \mathcal O(\SL_q(2))^{t,\alpha},\quad
\begin{pmatrix} a & b \\ c & d  \end{pmatrix} \mapsto \begin{pmatrix} a\otimes g & b\otimes g \\ c\otimes g & d \otimes g \end{pmatrix}
\end{equation*}
which is easily seen to be an isomorphism. Note that the above cocentral action also induces a cocentral action of $\mathbb Z_2$ on $\mathcal O(B_q)$, the quotient of $\mathcal O(\SL_q(2))$ by the ideal generated by $c$, and that similarly
$\mathcal O(B_{-q})$ is a graded twist of $\mathcal O(B_q)$. In addition, suppose that $K$ contains the imaginary unit $i$. Since
$$
\begin{pmatrix} a & b \\ c & d  \end{pmatrix} \mapsto \begin{pmatrix} i & 0 \\0 & -i\end{pmatrix}
$$
extends to an algebra homomorphism $\Pol(\SL_q(2)) \to K$, the pair $(p,\alpha)$ becomes an almost adjoint invariant cocentral action of $\mathbb Z_2$ on $\mathcal O(\SL_q(2))$.
\end{example}

\begin{example}\label{ex-be}
The example of  $\mathcal O(\SL_q(2))$ has the following natural generalization. Let $E\in \GL(n,K)$ and consider the Hopf algebra $\mathcal B(E)$ defined by Dubois-Violette and Launer \cite{MR1068703}: $\mathcal B(E)$ is the algebra generated by $a_{ij}$, $1 \leq i,j \leq n$, subject to the relations
$E^{-1}a^tEa=I_n=aE^{-1}a^tE$, where $a$ is the matrix $(a_{ij})$ (for an appropriate matrix $E_q$, one gets  $\mathcal O(\SL_q(2))$, see \cite{MR1068703}). Just as before, there is a cocentral Hopf algebra map
\begin{equation*}
 p\colon \mathcal B(E) \rightarrow K \mathbb Z_2,\quad
 a_{ij} \mapsto \delta_{ij}g.
\end{equation*}
Consider now a matrix $Q \in \GL(n,K)$ such that $Q^2=\pm I_n$  and $Q^tEQ=E$. Such a matrix defines an involutive Hopf algebra automorphism $\alpha_Q$ of $\mathcal B(E)$ with $\alpha_Q(a)=QaQ^{-1}$, and we get an almost adjoint invariant cocentral action $(p,\alpha_Q)$ of $\mathbb Z_2$ on $\mathcal B(E)$. Again a direct verification provides a Hopf algebra morphism
\begin{equation*}
 \mathcal B(EQ^{-1})=\mathcal B(EQ) \rightarrow \mathcal B(E)^{t, \alpha_Q},
 \quad a_{ij} \mapsto a_{ij} \otimes g
\end{equation*}
which is shown to be an isomorphism using Lemma \ref{iso-graded-twisted}.
\end{example}

\begin{example}\label{ex-freeproduct}
To finish the section with a cocentral action that is not almost adjoint, we examine an example related to free products. Let $A$ be a Hopf algebra endowed with a cocentral surjective Hopf algebra map $p \colon A \rightarrow K \mathbb Z_2$. Consider the free product Hopf algebra $A*A$ (as an algebra, it is the free product of $A$ with itself in the category of unital $K$-algebras). It admits the cocentral Hopf algebra map $A*A \rightarrow K\mathbb Z_2$ whose restriction to each copy is $p$, that we still denote $p$. Let $\alpha$ be the Hopf algebra automorphism of $A*A$ that exchanges the two copies of~$A$. We get a cocentral action $(p, \alpha)$ of $\mathbb Z_2$ on $A*A$, and hence a graded twisting $(A*A)^{t,\alpha}$. The simple corepresentations of $(A*A)^{t,\alpha}$ can be labeled by the reduced words on those of $A$ (since the same is true for the free product, see \cite{MR3413872}), and the fusion rules can be computed from those of $A$, see Proposition \ref{prop:grothendieck}.

Now let $A= \mathcal B(E)$ as in the previous example. Put $F=E^{t}E^{-1}$, and consider the universal cosovereign Hopf algebra $\Hpf(F)$ (see \cite{MR2302731}, these are the Hopf algebraic generalizations of the universal compact quantum groups from \cite{MR1382726}), which is the algebra generated by $u_{ij}$, $v_{ij}$, $1\leq i,j\leq n$, subject to the relations
$$uv^t=I_n=v^tu, \ Fu^tF^{-1}v=I_n= vFu^tF^{-1}$$
where $u$, $v$ are the matrices $(u_{ij})$, $(v_{ij})$ respectively. We can then directly verify that the correspondence
\begin{equation*}
u \mapsto a^{(1)} \otimes g, \quad v  \mapsto E^t a^{(2)} (E^{-1})^t \otimes g
\end{equation*}
defines a Hopf algebra homomorphism $\Hpf(F) \rightarrow (\mathcal B(E)*\mathcal B(E))^{t, \alpha}$, where the superscript refers to the numbering of copies inside the free product. To prove that this is an isomorphism, one constructs a cocentral action $(q,\beta)$ of $\mathbb Z_2$ on $\Hpf(F)$:
the cocentral Hopf algebra map is $q \colon \Hpf(F) \rightarrow K$, $u_{ij}, v_{ij} \mapsto \delta_{ij} g$, and the $\mathbb Z_2$-action on $\Hpf(F)$ is given by $\beta(u)= (E^{-1})^tvE^t$ and $\beta(v)=E^tu(E^{-1})^t$.
One then checks the existence of a Hopf algebra map
\begin{equation*}
 \mathcal B(E)*\mathcal B(E) \rightarrow \Hpf(F)^{t,\beta}, \quad
a^{(1)} \mapsto u \otimes g, \quad  a^{(2)} \mapsto (E^{-1})^tvE^t\otimes g
\end{equation*}
and using Lemma \ref{iso-graded-twisted} we conclude that $\Hpf(F)$ is a graded twist of $\mathcal B(E) * \mathcal B(E)$. The details are left to the reader.

It seems that the graded twisting picture is the simplest way to describe the corepresentations of $\Hpf(F)$ \citelist{\cite{MR1484551}\cite{MR2302731}\cite{MR2853627}\cite{MR3413872}} when $F$ is of type $E^{t}E^{-1}$.
\end{example}

Several other examples will be examined, at the compact quantum group level, in the next sections.

\section{Compact quantum groups}
\label{sec:cpt-q-grps}

In this section we specialize to the case of compact quantum groups. In particular, we assume that $K=\C$. We refer the reader to~\cite{MR3204665} for basic terminology.

\subsection{Twisting regular functions of compact quantum groups}

We work with Hopf algebras $A=\Pol(G)$ of regular functions of compact quantum groups. This simply means that $A$ is assumed to be a Hopf $*$-algebra such that every finite dimensional corepresentation is unitarizable, or equivalently, $A$ is generated as an algebra by matrix coefficients of finite dimensional unitary corepresentations, see~\citelist{\cite{MR1492989}*{Section~11.3.1}\cite{MR3204665}*{Section~1.6}}.

We denote the dual algebra of $\Pol(G)$ by $\U(G)$. It is a $*$-algebra, with the $*$-structure defined by $\omega^*(a)=\overline{\omega(S(a)^*)}$ for $\omega\in \U(G)$ and $a\in\Pol(G)$. More generally, we write $\U(G^n)$ for the dual of $\Pol(G)^{\otimes n}$. We denote by $\hat\Delta\colon \U(G)\to\U(G\times G)$ the `coproduct' map which is the dual of the product $\Pol(G)\otimes\Pol(G)\to\Pol(G)$. We also write $\Rep G$ instead of $\Corep(\Pol(G))$, and $\Ch(G)$ instead of $\Ch(\Corep(\Pol(G)))$.

When dealing with compact quantum groups, we require invariant cocentral actions $(p,\alpha)$ to be $*$-preserving. In other words, we assume that the action satisfies $\alpha_g(a)^*=\alpha_g(a^*)$. Note that we automatically have $a^*\in \Pol(G)_{g^{-1}}$ for $a\in \Pol(G)_g$. Indeed, such $a$ is of the form $(\omega_{\xi,\eta}\otimes\iota)(U)$ for a unitary representation $U \in (\Rep G)_g$ and $\xi, \eta \in H_U$, where we put $\omega_{\xi,\eta}(T) = (T \xi, \eta)$ for $T \in B(H_U)$. The involution is characterized by $(\omega_{\xi,\eta}\otimes\iota)(U)^* = (\omega_{\eta,\xi}\otimes S)(U)$, and $S(\Pol(G)_g) = \Pol(G)_{g^{-1}}$ (see remark after Lemma~\ref{lem:gradings}) implies $a^* \in \Pol(G)_{g^{-1}}$. Moreover, the corresponding cocentral homomorphism $\pi \colon \Pol(G) \to \C\Gamma$ becomes a $*$-homomorphism with respect to the standard $*$-structure $g^* = g^{-1}$ on $\C\Gamma$. Indeed, if $a\in\Pol(G)_g$, then $\pi(a)=\eps(a)g$, and since $\eps$ is a $*$-homomorphism, we get $\pi(a)^*=\eps(a^*)g^{-1}=\pi(a^*)$.

In this case the twisted Hopf algebra $\Pol(G)^{t,\alpha}$ is again a Hopf $*$-algebra, with the $*$-structure inherited from $A\rtimes_\alpha\Gamma$. Hence it is the algebra of regular functions on a compact quantum group, which we denote by $G^{t,\alpha}$.

While in the noncoamenable case it is hard to relate the C$^*$-norms on $\Pol(G^{t,\alpha})$ to those on~$\Pol(G)$, the reduced norms can be easily related as follows.

\begin{proposition}\label{prop:reduced-algebra}
The C$^*$-algebra $C_r(G^{t,\alpha})$ coincides with the norm closure of $\Pol(G)^{t,\alpha}$ in the reduced crossed product $C_r(G)\rtimes_{\alpha,r}\Gamma$.
\end{proposition}

\bp The Haar state on $\Pol(G)\rtimes_\alpha\Gamma$ coincides with the composition of the Haar state on $\Pol(G)$ with the canonical conditional expectation $\Pol(G)\rtimes_\alpha\Gamma\to\Pol(G)$, $a\otimes g\mapsto \delta_{g,e}a$. It follows that the completion of $\Pol(G)\rtimes_\alpha\Gamma$ in the corresponding GNS-representation coincides with $C_r(G)\rtimes_{\alpha,r}\Gamma$. Since the Haar state on $\Pol(G)^{t,\alpha}$ is the restriction of that on $\Pol(G)\rtimes_\alpha\Gamma$, we get the assertion.
\ep

As is common in operator algebra, from now on we denote the elements $a\otimes g\in \Pol(G)\rtimes_\alpha\Gamma$ by $a\lambda_g$.

Let us now say a few words about the case of an almost adjoint action. Denote by $H^1(\hat G;\T)$, resp.~by $H^1_G(\hat G;\T)$, the subgroup of $H^1(\Pol(G))$, resp.~of $H^1_\ell(\Pol(G))$, consisting of unitary elements. Then, when talking about compact quantum groups, we assume that an almost adjoint action is given by a pair $(\phi,\mu)$ such that $\phi$ takes values in $H^1(\hat G;\T)$, while $\mu$ is a cocycle with values in $H^1_G(\hat G;\T)$. Given such an invariant cocentral action, we get a cocycle $c\in Z^3(\Ch(G);\T)$ by Proposition~\ref{prop:3-cocycle} and a pseudo-$2$-cocycle $\sigma\in \U(G\times G)$ by Theorem~\ref{thm:pseudo-2-cocycle}. This pseudo-cocycle is unitary, and in our present notation identity \eqref{eq:associator2} becomes
\begin{equation}\label{eq:associator3}
(\iota\otimes\Dhat)(\sigma^{-1})(1\otimes\sigma^{-1})(\sigma\otimes1)(\Dhat\otimes\iota)(\sigma)=c\ \ \text{in}\ \ \U(G\times G\times G),
\end{equation}
where we view $c$ as an element of $\U(G^3)$ using the embedding of the function algebra on $\Ch(G)^3$ into $\U(G^3)$ dual to $p^{\otimes 3}\colon\Pol(G)^{\otimes 3}\to(\C\Ch(G))^{\otimes 3}$, and we omit the notation for the convolution product on $\U(G^n)$.

\begin{rem}\label{rem:abelian-case-and-lifting}
The compact group $T=H^1(\hat G;\T)$ is a closed subgroup of $G$: the surjective Hopf $*$-algebra homomorphism $\Pol(G)\to\Pol(T)$ maps $a\in\Pol(G)$ into the function that takes value $u(a)$ at $u\in T$. If $\Ch(G)$ is abelian, then its dual $H^1_G(\hat{G}; \T)$ is a central subgroup of $T$, and the pseudo-$2$-cocycle $\sigma$ is in $\U(T \times T)$. If furthermore $T$ is abelian, then $\sigma$ can be regarded as a $2$-cochain on $\hat T$ whose coboundary is the inverse of the lift of the $3$-cocycle $c$ on $\Ch(G)$ to~$\hat T$.
\end{rem}

\subsection{Orthogonal, hyperoctahedral and unitary quantum groups}

In this section we consider some concrete examples of our twisting procedure.

\begin{example}[Orthogonal quantum groups] \label{ex:orth}
Let $F \in M_n(\C)$ be a matrix satisfying $F\overline{F} = c I_n$, $c \in \mathbb R^*$. Recall that the free orthogonal group $\mathrm{O}_F^+$ \cites{MR1382726, MR1378260} is defined
as follows. The algebra $\Pol(\mathrm{O}_F^+)$ is the universal unital $*$-algebra generated by entries of an $n$-by-$n$ matrix $U=(u_{ij})_{i,j}$ subject to the relations $U^*U=UU^*=I_n$ and $F\bar UF^{-1}=U$, where $\bar U=(u^*_{ij})_{i,j}$.
 The C$^*$-envelope of $\Pol(\mathrm{O}^+_F)$ is denoted by $A_o(F)$.

It is easy to check that, as a Hopf algebra, $\mathcal O(\mathrm{O}^+_F)$ is the Hopf algebra $\mathcal B((F^{-1})^t)$ in Example~\ref{ex-be}. The considerations there can be easily adapted: let  $Q \in \GL(n,\mathbb C)$ be a unitary matrix such that $Q^2=\pm I_n$, $Q^t\bar FQ=\bar F$. We  obtain an almost adjoint invariant cocentral action $\alpha_Q$ of $\ZZ_2$ on $\Pol(\mathrm{O}^+_F)$, corresponding to the group homomorphism $\ZZ_2 \to H^1(\hat{\mathrm{O}}_F^+; \T) / H^1_{\mathrm{O}_F^+}(\hat{\mathrm{O}}_F^+; \T)$ sending the nontrivial element to the class of $Q$,
 with $\alpha_Q(U)= QUQ^{-1}$, and a twisted compact quantum group $(\mathrm{O}_F^+)^{t,\alpha_Q}$. Similarly to Example \ref{ex-be}, the compact quantum group $(\mathrm{O}_F^+)^{t,\alpha_Q}$ is isomorphic to $\mathrm{O}_{FQ^t}^+$.

If $F=I_n$, then $\mathrm{O}_F^+$ is denoted by $\mathrm{O}_n^+$, and the C$^*$-envelope of $\Pol(\mathrm{O}^+_n)$ is denoted by $A_o(n)$.
A matrix $Q$ as before is a real orthogonal matrix with $Q^2=\pm I_n$, and there are two cases:
\begin{enumerate}
 \item $Q^2=I_n$: in this case the matrix $Q$ is symmetric and $\mathrm{O}^+_Q$ is isomorphic to $\mathrm{O}^+_n$.
\item $Q^2=-I_{n}$: in this case necessarily $n$ is even and the matrix $Q$ is anti-symmetric. Without changing the isomorphism class of $\mathrm{O}^+_Q$, we can assume \cite{MR2202309} that $Q$ is the matrix
$$
J_{2m}=\left(\begin{array}{cc}0_m & I_m \\-I_m & 0_m\end{array}\right).
$$
The category $\Rep (\mathrm{O}^+_n)^{t,\alpha_Q}$ is the nontrivial twist $(\Rep \mathrm{O}^+_n, \Phi)$ for the associator $\Phi$ acting on the three-fold tensor product of irreducible representations $U_{a/2} \otimes U_{b/2} \otimes U_{c/2}$ by $(-1)^{a+b+c}$, where $U_{a/2}$ is the irreducible representation of spin $a/2$ for $a \in \N$.
\end{enumerate}
For $n=2$, we have $A_o(J_{2})\simeq C(\SU(2))$, while $A_o(2) \simeq C(\SU_{-1}(2))$, so we recover the twisting of $\SU(2)$ into $\SU_{-1}(2)$ in \cite{MR3340190}.
\end{example}

The above procedure for twisting $\mathrm{O}_{2m}^+$ works for any quantum subgroup $\ZZ_2 \subset H \subset \mathrm{O}_{2m}^+$ such that $J_{2m} \in H^1(\hat{H}; \T)$. We will have a brief look at two examples of this situation.

\begin{example}[Half-liberated orthogonal quantum groups]
Recall \cite{MR2554941} that the half-liberated orthogonal quantum group $\mathrm{O}_n^*$ is defined as follows: $\mathcal O(\mathrm{O}_n^*)$ is the quotient of $\mathcal O(\mathrm{O}_n^+)$ by the relations
$abc=cba$, $a,b,c \in \{u_{ij}\}$.
We get a twisted half-liberated orthogonal quantum group $\mathrm{O}_{J_{2m}}^*= (\mathrm{O}_{2m}^*)^{t,\alpha_{J_{2m}}}$, with $\mathcal O(\mathrm{O}_{J_{2m}}^*)$ the quotient of $\mathcal O(\mathrm{O}_{J_{2m}}^+)$ by the same relations as above.
\end{example}

\begin{example}[Hyperoctahedral quantum groups]
Recall \cites{MR2096666, MR2376808}
that the hyperoctahedral quantum group $\mathrm{H}_n^+$ is defined as follows: $\mathcal O(\mathrm{H}_n^+)$ is the quotient of $\mathcal O(\mathrm{O}_n^+)$ by the relations $u_{ij}u_{ik}=0=u_{ji}u_{ki}$ if $j\not=k$. We obtain a twisted hyperoctahedral quantum group $\mathrm{H}_{J_{2m}}^+$, with $\mathcal O(\mathrm{H}_{J_{2m}}^+)$ the quotient of  $\mathcal O(\mathrm{O}_{J_{2m}}^+)$ by the relations $u_{ij}u_{kl}=0$ if $i+k=m+1$ and $j+l\not=m+1$, or $i+k\not=m+1$ and $j+l=m+1$. Similar considerations also work for the hyperoctahedral series $\mathrm{H}_n^{(s)} \subset \mathrm{H}_n^* \subset \mathrm{H}_{n}^{[s]}$ considered in \cite{MR2718205}.
\end{example}

\begin{example}[Unitary quantum groups]
Recall that the free unitary quantum group $\mathrm{U}_n^+$ \cite{MR1316765} is defined as follows: the algebra $\Pol(\mathrm{U}_n^+)$ is the universal unital $*$-algebra generated by entries of an $n$-by-$n$ matrix $U=(u_{ij})_{i,j}$ subject to the relations $U^*U=UU^*=I_n$ and $\bar U U^t=I_n=U^t \bar U$, where $\bar U=(u^*_{ij})_{i,j}$.
 The C$^*$-envelope of $\Pol(\mathrm{U}^+_n)$ is denoted by $A_u(n)$.

We have $\mathcal O(\mathrm{U}_n^+) \simeq \Hpf(I_n)$ as Hopf algebras, where $\Hpf(I_n)$ is the Hopf algebra of Example~\ref{ex-freeproduct}, and the constructions there can be adapted easily, showing that $\mathcal O(\mathrm{U}_n^+)$ is a graded twist of $\mathcal O(\mathrm{O}_n^+)*\mathcal O(\mathrm{O}_n^+)$.
\end{example}

The above construction works for any quantum group $H$ with $\mathbb Z_2 \subset H \subset \mathrm{O}_n^+$, to produce a graded twisting of $\mathcal O(H)*\mathcal O(H)$.
To conclude the section, we briefly mention two more examples.

\begin{example}[Half-liberated unitary quantum groups] We put $H=\mathrm{O}_n^*$ above. The twisted version of  $\mathcal O(\mathrm{O}_n^*)*\mathcal O(\mathrm{O}_n^*)$ we obtain is $\mathcal O(\mathrm{U}_n^*)$, defined in \cite{MR2835874}, i.e.,~the quotient of $\mathcal O(\mathrm{U}_n^+)$ by the relations $ab^*c=cb^*a$, $a,b,c \in \{u_{ij}\}$.
\end{example}

\begin{example}
Now put $H=\mathrm{O}_n$. Then the twisted version of  $\mathcal O(\mathrm{O}_n)*\mathcal O(\mathrm{O}_n)$ we obtain is $\mathcal O(\mathrm{U}_n^{\#})$, recently defined in \cite{arXiv:1505.00646}, i.e.,~the quotient of $\mathcal O(\mathrm{U}_n^+)$ by the relations $ab^*=ba^*$, $a,b \in \{u_{ij}\}$.
\end{example}

\subsection{Twisting and coamenability}

Let us return to the general case of compact quantum groups. Recall that a compact quantum group $G$ is called \emph{coamenable} if the counit on $\Pol(G)$ extends to a bounded linear functional on $C_r(G)$, in which case the C$^*$-algebra $C_r(G)$ coincides with the C$^*$-envelope of $\Pol(G)$ and we simply write $C(G)$. In this section we explore implications of coamenability for twisting.

\smallskip

We start with the following simple observation.

\begin{proposition}\label{prop:enveloping}
 Let $G$ be a coamenable compact quantum group and $(p,\alpha)$ be an invariant cocentral action of $\Gamma$ on $\Pol(G)$. Then the image of $p$ is contained in $\C\Gamma_0$ for an amenable subgroup $\Gamma_0$ of $\Gamma$, and the compact quantum group $G^{t,\alpha}$ is coamenable.
\end{proposition}

\begin{proof}
First of all observe that the cocentral $*$-homomorphism $\Pol(G) \to \C\Ch(G)$ realizes the dual quantum group of $\Ch(G)$ as a quantum subgroup of $G$. Since coamenability passes to quantum subgroups, the quantum group $\widehat{\Ch(G)}$ must be coamenable, hence $\Ch(G)$ is amenable. The $\Gamma$-grading on $\Pol(G)$ is defined by a homomorphism $\Ch(G)\to\Gamma$. Hence the image $\Gamma_0$ of $\Ch(G)$ in $\Gamma$ is amenable.

By Proposition~\ref{prop:reduced-algebra}, we have $C_r(G^{t,\alpha})\subset C(G)\rtimes_r\Gamma_0$.
Since the counit map $\couni \colon C(G) \to \C$ is $\Gamma$-equivariant, it induces a $*$-homomorphism $\couni \otimes \iota\colon C(G) \rtimes_r \Gamma_0 \to \C \rtimes_r \Gamma_0 = C^*_r(\Gamma_0)$. The counit on $C_r(G^{t,\alpha})$ is the composition of this homomorphism with the counit $C^*_r(\Gamma_0)\to\C$. Hence it is bounded.
\end{proof}

If $\alpha$ is almost adjoint, the above result follows also from the fact that $G$ and $G^{t,\alpha}$ have the same fusion rules and the classical dimension functions, see \cite{MR1679171} or \cite{MR3204665}*{Section~2.7}.

\smallskip

Our goal is to show that coamenability puts restrictions on applicability of our twisting procedure. We start with an auxiliary categorical consideration. Let $\A$ be an essentially small rigid C$^*$-tensor category. (We assume that the C$^*$-tensor categories that we consider in this section have simple units and are closed under finite direct sums and subobjects.) Consider the set~$I_\A$ of isomorphism classes of simple objects in $\A$. For every $s\in I_\A$ choose a representative~$U_s$. For every object~$U$ let $\Gamma_U\in B(\ell^2(I_\A))$ be the operator such that its matrix coefficient corresponding to $s,t\in I$ is the multiplicity of $U_s$ in $U\otimes U_t$.

Given another rigid C$^*$-tensor category $\BB$ and a unitary tensor functor $F\colon \A\to \BB$, the functor $F$ is called \emph{amenable} if the fusion ring of $\A$ is amenable (see~\cite{MR1644299}) and $\|\Gamma_U\|=d^\BB(F(U))$ for every object $U$ in $\A$, where $d^\BB$ denotes the quantum dimension. It is shown in~\cite{arXiv:1405.6572} that once the fusion ring of $\A$ is amenable, there exists a universal amenable functor $\Pi\colon\A\to\PP$. Furthermore, such a universal amenable functor can be constructed as follows. Assume $\A$ is a subcategory of a rigid C$^*$-tensor category~$\BB$ such that $\|\Gamma_U\|=d^\BB(U)$ for all objects $U$ in $\A$. For every object $U$ in $\A$ choose a standard solution $(R_U,\bar R_U)$ of the conjugate equations for $U$ in $\A$. Then there exists a unique positive automorphism $a_U$ of $U$ in $\BB$ such that $((\iota\otimes a_U^{1/2})R_U,(a_U^{-1/2}\otimes\iota)\bar R_U)$ is a standard solution of the conjugate equations for $U$ in~$\BB$. Then as $\PP$ we can take the C$^*$-tensor subcategory of $\BB$ generated by $\A$ and the morphisms~$a_U$, and as $\Pi$ we can take the embedding functor. In fact, this description of $\Pi\colon \A\to\PP$ was given in~\cite{arXiv:1405.6572}*{Section~4} for strict C$^*$-tensor categories. But since every category is equivalent to a strict one, it is clear that it remains true in general.

Consider now a coamenable compact quantum group $G$. Coamenability means exactly that the canonical fiber functor $\Rep G\to\Hilb_f$ is amenable. In this case it can be shown, see~\cite{arXiv:1405.6574}*{Theorem~2.1}, that the universal amenable unitary tensor functor is the restriction functor $\Rep G\to \Rep K$, where $K\subset G$ is the maximal closed quantum subgroup of Kac type, that is, $\Pol(K)$ is the quotient of $\Pol(G)$ by the two-sided ideal generated by the elements $a-S^2(a)$ for $a \in \Pol(G)$.

Before we formulate the main result of this section, let us also introduce another piece of notation. Given a cocycle $c\in Z^3(\Ch(G);\T)$, denote by $(\Rep G)^c$ the monoidal category $\Rep G$ with the new associativity morphisms such that
$(U\otimes V)\otimes W\to U\otimes (V\otimes W)$ is the scalar morphism $c(g,h,k)$ for $[U]=g$, $[V]=h$ and $[W]=k$. We can  also assume that the cocycle~$c$ is normalized,
$$
c(e,g,h)=c(g,e,h)=c(g,h,e)=1,
$$
without loss of generality.

\begin{theorem}
Let $G$ be a coamenable compact quantum group with maximal closed quantum subgroup $K$ of Kac type. Assume we are given a $3$-cocycle $c\in Z^3(\Ch(G);\T)$. Then $\Ch(G)$ is a quotient of $\Ch(K)$ and the restriction functor $(\Rep G)^c\to(\Rep K)^c$ is a universal amenable unitary tensor functor.
\end{theorem}

\bp The first statement is obvious, since the cocentral Hopf algebra homomorphism $\Pol(G)\to\C\Ch(G)$ factors through $\Pol(K)$ by the maximality of $K$.

Next, consider the Woronowicz character $\rho\in\U(G)$. For every representation $U\in B(H_U)\otimes\Pol(G)$ of $G$ it defines an operator $\rho_U=(\iota\otimes\rho)(U)$ on $H_U$. Then $\Rep K$ can be described as the C$^*$-tensor subcategory of $\Hilb_f$ generated by $\Rep G$ and the morphisms $\rho_U$, see the proof of~\cite{arXiv:1405.6574}*{Theorem~2.1}. In order to connect this with the above discussion of the construction of amenable functors, recall that $\rho_U$ are exactly the morphisms $a_U$ we discussed above for the categories $\Reg G\subset\Hilb_f$ and a suitable choice of standard solutions in $\Rep G$.

Consider the restriction functor $(\Rep G)^c\to(\Rep K)^c$. This functor is amenable, since the new associativity morphisms change neither fusion rules nor quantum dimensions. In order to show that this functor is universal we have to find the morphisms $a_U$ relating solutions of the conjugate equations in $(\Rep G)^c$ and $(\Rep K)^c$, and then show that $(\Rep K)^c$ is generated by $(\Rep G)^c$ and these morphisms. Given an irreducible representation $U$ of $G$ of degree $g\in\Ch(G)$ and a standard solution $(R_U,\bar R_U)$ of the conjugate equations for $U$ in $\Rep G$, the pair $(R_U^c,\bar R^c_U)=(R_U,c(g^{-1},g,g^{-1})\bar R_U)$ forms a standard solution for $U$ in $(\Rep G)^c$. From this it becomes clear that the morphism $a_U$ in $(\Rep K)^c$ such that $((\iota\otimes a_U^{1/2})R^c_U,(a_U^{-1/2}\otimes\iota)\bar R^c_U)$ is a standard solution of the conjugate equations for $U$ in $(\Rep K)^c$ is the same as in the case of the trivial cocycle, that is,~$a_U=\rho_U$.

It remains to show that $(\Rep K)^c$ is generated by $(\Rep G)^c$ and the morphisms $\rho_U$. Take representations $U$ and $V$ of $G$. Then the fact that $\Rep K$ is generated as a C$^*$-tensor category by $\Rep G$ and the morphisms $\rho_W$ means that any morphism $U|_K\to V|_K$ can be written as a linear combination of compositions of morphisms $U'\to V'$ in $\Rep G$ and morphisms of the form $\iota_{X}\otimes\rho_Y\otimes\iota_Z$. But the same compositions makes sense in $(\Rep K)^c$, with any distribution of brackets on $X\otimes Y\otimes Z$. Hence the tensor category generated by $(\Rep G)^c$ and $\rho_W$ contains all morphisms $U\to V$ in $(\Rep K)^c$. This proves the assertion.
\ep

The above theorem can be used to simplify the considerations in~\cite{arXiv:1405.6574}*{Section~3.1}. More importantly for the present work, it gives the following.

\begin{corollary} \label{cfib}
Under the assumptions of the theorem, a unitary fiber functor $F\colon(\Rep G)^c\to\Hilb_f$ such that $\dim F(U)=\dim U$ for all $U$ exists if and only if there exists a unitary fiber functor $(\Rep K)^c\to\Hilb_f$.
\end{corollary}

\bp Since a unitary fiber functor $F\colon(\Rep G)^c\to\Hilb_f$ such that $\dim F(U)=\dim U$ is amenable, by universality it must factor through $(\Rep K)^c$, in a unique up to a natural unitary monoidal isomorphism way. This proves the corollary in one direction, and the other direction is obvious. (It is worth recalling that since $(\Rep K)^c$ is amenable, any unitary fiber functor $(\Rep K)^c\to\Hilb_f$ is dimension-preserving.)
\ep

Therefore if, for example, there are no unitary fiber functor $(\Rep K)^c\to\Hilb_f$, then for any twist $G^{t,\alpha}$ of $G$ we cannot have the category $(\Rep G)^c$ as the representation category of $G^{t,\alpha}$.

\subsection{Drinfeld--Jimbo deformations of compact Lie groups}

%We now apply the above results to the $q$-deformations.

Consider a compact simply connected semisimple Lie group $G$ and, for $q>0$, its $q$-deforma\-tion~$G_q$. (In fact, with a suitable definition of $q$-deformation, everything what follows remains true in the non-simply-connected case.) We denote the coproduct on $\U(G_q)$ by $\Dhat_q$.

Let $T$ be the maximal torus in $G$ which remains undeformed in $G_q$. The center $Z(G)$ of $G$ is well-known to be contained in $T$, and its dual is~$P/Q$, where $P$ and $Q$ are the weight and root lattices, respectively. It is also known that for $q\ne1$ we have
$$
H^1(\hat G_q;\T)=T\ \ \text{and}\ \ \Ch(G_q)=\widehat{Z(G)}=P/Q.
$$
In particular, we are in the setting of Remark~\ref{rem:abelian-case-and-lifting}, so our construction of pseudo-$2$-cocycles simply produces particular $2$-cochains on $\hat T=P$ with coboundary living on $P/Q$. The following result shows that there is essentially no other way of constructing pseudo-$2$-cocycles with coboundary living on $\Ch(G_q)$.

\begin{theorem}\label{thm:3coc-2pseudo}
With the above notation, assume $q>0$, $q\ne1$, and $c\in Z^3(P/Q;\T)$. Then the following conditions are equivalent:
\begin{enumerate}
\item there is a unitary fiber functor $F\colon(\Rep G_q)^c\to\Hilb_f$ such that $\dim F(U)=\dim U$ for all finite dimensional unitary representations $U$ of $G_q$;
\item there exists a unitary $\sigma\in\U(G_q\times G_q)$ such that
\begin{equation} \label{eassoc}
(\iota\otimes\Dhat_q)(\sigma^{-1})(1\otimes\sigma^{-1})(\sigma\otimes1)(\Dhat_q\otimes\iota)(\sigma)=c;
\end{equation}
\item there exists a unitary $\sigma$ satisfying~\eqref{eassoc} obtained by the construction in Theorem~\ref{thm:pseudo-2-cocycle} (for a quotient $\Gamma$ of $P/Q$ and an action of $\Gamma$ on $C_r(G_q)$);
\item the lift of $c$ to $P$ is a coboundary;
\item the cocycle $c$ vanishes on $\wedge^3(P/Q)\subset H_3(P/Q;\ZZ)$.
\end{enumerate}

Furthermore, if these conditions are satisfied, then all unitaries $\sigma$ satisfying \eqref{eassoc} have the form
$$
\sigma=(u\otimes u)f\Dhat_q(u)^{-1},
$$
where $u\in\U(G_q)$ is a unitary element and $f$ is a $\T$-valued $2$-cochain on $P$ such that $\partial f=c^{-1}$.
\end{theorem}

\bp The equivalence of (i) and (ii) is a simple consequence of the definitions and is well-known to be true for any compact quantum group and a unitary associator $c$ on its representation category.

Since $q\ne1$, the maximal closed quantum subgroup of $G_q$ of Kac type is $T$, see~\cite{MR2335776}*{Lemma~4.10} or~\cite{arXiv:1405.6574}*{Theorem~3.1}. Therefore by Corollary~\ref{cfib}, if (i) is satisfied then there exists a unitary fiber functor $(\Rep T)^c\to\Hilb_f$, so by the analogue of equivalence between (i) and (ii) for $T$ instead of $G_q$ we conclude that $c$ is a coboundary on $P$. Thus (i) implies (iv).

The equivalence of (iv) and (v) is proved in~\cite{MR3340190}*{Corollary~A.4}. In \cite{MR3340190}*{Theorem~3.1 and Corollary~A.4} it is also shown that (iv) implies (iii). Finally, (iii) obviously implies~(ii).

\smallskip

Assume now that conditions (i)--(v) are satisfied and choose a $\T$-valued $2$-cochain $f$ on $P$ such that $\partial f=c^{-1}$. Consider the quantum group $G^f_q$ obtained by twisting the product on $\Pol(G_q)$ by $f$, or in other words, by replacing the coproduct on $\U(G_q)$ by $\Dhat^f_q=f\Dhat_q(\cdot)f^{-1}$. Take any unitary $\sigma$ satisfying \eqref{eassoc}. Then, since $f$ also satisfies \eqref{eassoc}, the unitary $\Omega=\sigma f^{-1}$ is a dual $2$-cocycle on $G^f_q$, that is,
$$
(\Omega\otimes1)(\Dhat^f_q\otimes\iota)(\Omega)=(1\otimes\Omega)(\iota\otimes\Dhat^f_q)(\Omega).
$$
By~\cite{arXiv:1405.6574}*{Corollary~3.3}, any such cocycle is cohomologous to a cocycle on $P$, that is, there exist a unitary $u\in\U(G_q^f)=\U(G_q)$ and a cocycle $\omega\in Z^2(P;\T)$ such that $\Omega=(u\otimes u)\omega\Dhat^f_q(u)^{-1}$. Then $\sigma=(u\otimes u)\omega f\Dhat_q(u)^{-1}$ and $\partial(\omega f)=c^{-1}$.
\ep

The above theorem shows that if $\wedge^3(P/Q) \neq 0$, e.g., $G = \mathrm{SU}(2)^3$, any choice of $c$ which is nontrivial on $\wedge^3(P/Q)$ leads to the absence of unitary fiber functors on $(\Rep G_q)^c$ preserving the classical dimension. This gives a partial answer to the question of existence of fiber functors on $(\Rep G_q)^c$ raised in~\cite{MR3340190}.

It is still possible, however, that there are fiber functors of a different dimension. Of course, if such functors exist, they are not defined by pseudo-cocycles. Let us show that at least for some $q$ such functors do not exist.

\begin{corollary}\label{cor:kwnotqg}
Assume that a cocycle $c\in Z^3(P/Q;\T)$ does not vanish on $\wedge^3(P/Q)$. Then there exists an interval $(a,b)\subset\R$ containing $1$ such that for all $q\in(a,b)$, $q\ne1$, there are no unitary fiber functors $(\Rep G_q)^c\to\Hilb_f$.
\end{corollary}

\bp Choose a self-conjugate faithful unitary representation $U$ of $G$. Denote by the same symbol $U$ the corresponding representation of $G_q$. Choose an interval $(a,b)\subset\R$ such that $\dim_q U<\dim U+1$ for $q\in(a,b)$. Then for any unitary fiber functor $F\colon(\Rep G_q)^c\to\Hilb_f$ we have $\dim F(U)=\dim U$. But there exists only one dimension function on the representation ring of $G$ with value $\dim U$ on the class of $U$, see the proof of \cite{MR1014926}*{Theorem~19}. Therefore~$F$ preserves the classical dimension, and by the above theorem we know that such a functor does not exist for $q\ne1$.
\ep

\bigskip

\section{Compact groups}
\label{sec:cpt-grp-twist}

In this section we consider genuine compact groups $G$, and study the problem of describing closed quantum subgroups of $G^{t,\alpha}$.

\subsection{Quantum subgroups}
\label{sec:quotients}

Let us start with a short general discussion of quotients of~$A^{t,\alpha}$. Let $(p, \alpha)$ be an invariant cocentral action of $\Gamma$ on a Hopf algebra $A$.
The basic construction is the following. Let $I$ be a $\Gamma$-stable Hopf ideal of $A$ such that $I \subset \Ker(p)$, $\pi\colon A\to A/I$ be the quotient map, $\bar{p}$ be the induced cocentral homomorphism of Hopf algebras $\bar{p}\colon A/I \rightarrow K\Gamma$, and $\bar{\alpha}$ be the induced action $\bar{\alpha}\colon \Gamma \rightarrow \Aut(A/I)$ of $\Gamma$ on the quotient. Then $(\bar{p},\bar{\alpha})$ is again an invariant cocentral action of $\Gamma$ on $A/I$. The map $\pi \otimes \iota\colon A\rtimes_\alpha\Gamma\to (A/I)\rtimes_{\bar\alpha}\Gamma$ is a homomorphism of Hopf algebras, so by restriction it defines a Hopf algebra homomorphism $A^{t,\alpha} \rightarrow (A/I)^{t,\bar{\alpha}}$. As $\pi(A_g)=(A/I)_g$, this map is surjective, so $(A/I)^{t,\bar{\alpha}}$ is a quotient of $A^{t,\alpha} $.

Let us characterize the quotients of $A^{t,\alpha}$ that arise in this manner. For any $g \in \Gamma$, the restriction of $\alpha_g\otimes \iota$ defines a $\Gamma$-graded coalgebra automorphism of $A^{t,\alpha}$. (As we observed in Section~\ref{sec:abelian-groups}, this is an algebra automorphism if $\Gamma$ is abelian, but not in general.) We say that a subspace $X \subset  A^{t,\alpha}$ is $\Gamma$-stable if $(\alpha_g\otimes \iota)(X)\subset X$ holds for any $g \in \Gamma$.

\begin{proposition}\label{prop:correspideals}
Let $(p, \alpha)$ be an invariant cocentral action of $\Gamma$ on $A$, and
let $f\colon A^{t,\alpha} \rightarrow B$ be a surjective Hopf algebra map such that $J=\Ker(f)$ is a $\Gamma$-stable Hopf ideal with $J \subset \Ker(\tilde{p})$. Then there exists a $\Gamma$-stable Hopf ideal $I \subset A$ satisfying $I \subset \Ker(p)$, such that $(A/I)^{t,\bar{\alpha}}$ is isomorphic to $B$ as a Hopf algebra.
\end{proposition}

\begin{proof}
Using the $\Gamma$-graded coalgebra isomorphism $j$, we put $I = j^{-1}(J)$. Thus, an element $a = \sum_g a_g$ with $a_g \in A_g$ belongs to $I$ if and only if $\sum_g a_g \otimes g$ belongs to $J$. We claim that $I$ is a $\Gamma$-stable Hopf ideal of $A$ with $I \subset \Ker(p)$.

The first observation is that $I$ is a coideal, since $j$ is a coalgebra map and $J$ is a coideal. Next, by $\tilde{p} j = p$, $J \subset \Ker(\tilde{p})$ implies $I \subset \Ker(p)$. Since $j$ intertwines $\alpha_g$ with $\alpha_g\otimes\iota$, we also see that $I$ is $\Gamma$-stable.

In order to show that $I$ is also an ideal, let us note that since $J$ is a coideal contained in $\Ker(\tilde p)$, the map $J\to J\otimes K\Gamma $, $x\mapsto x_{(1)}\otimes\tilde p(x_{(2)})$, defines a coaction of $\Gamma$ on $J$. Hence $J$ decomposes into homogeneous components, so that if $a = \sum_g a_g\in A$, then the element $j(a)$ belongs to $J$ if and only if each $j(a_g) = a_g \otimes g$ belongs to $J$.
Combining the homogeneous decomposition $a = \sum_g a_g$ and the antipode formula $S(j(a_g)) = j(S(\alpha_{g^{-1}}(a_g)))$, we see that $I$ is stable under the antipode.

Now, suppose that $a = \sum_g a_g \in I$, and let $b=\sum_h b_h$, for $b_h \in A_h$, be another element. Then for any $g,h$, we have $a_g \otimes g \in J$, which implies that $(a_g  \otimes g)(\alpha_{g^{-1}}(b_h)\otimes h)= a_g b_h\otimes g h \in J$. Thus we obtain $a_g b_h \in I$, so that $I$ is a right ideal, hence a $\Gamma$-invariant Hopf ideal as claimed.

We can therefore form the Hopf algebra $(A/I)^{t,\bar{\alpha}}$ as above, and the canonical projection $\pi\colon A \rightarrow A/I$ induces a surjective Hopf algebra map $\pi \otimes \iota \colon A^{t,\alpha} \rightarrow (A/I)^{t,\bar{\alpha}}$. Looking at the homogeneous elements in the kernel, we see that $J$ is exactly the kernel of this projection.
\end{proof}

Let us now turn to the case of compact groups. So from now on we assume that $K=\C$ and $A = \Pol(G)$, the Hopf $*$-algebra of regular functions on a compact group $G$. Then the category $\Corep(A) = \Rep G$ is semisimple and has a commutative chain group, namely, the Pontryagin dual of the center $Z(G)$ of $G$. Therefore, since a $\Gamma$-grading is given by a homomorphism $\Ch(G)\to\Gamma$ and only depends on the (abelian) image of this homomorphism, the group $\Gamma$ can be assumed, without loss of generality, to be abelian.

We denote by $i\colon\hat\Gamma\to Z(G)$ the homomorphism dual to $\Ch(G)\to\Gamma$. A function $f \in \Pol(G)$ belongs to the homogeneous subspace $\Pol(G)_g$ for some $g \in \Gamma$ if and only if $f(i(\psi) x) = \psi(g) f(x)$ holds for any $\psi \in \hat{\Gamma}$ and any $x \in G$. Moreover, the decomposition of $f$ according to the grading $\Pol(G)=\oplus_{g \in \Gamma} \Pol(G)_g$ is given by
\begin{equation*}
\label{eq:homogen-decomp-of-func}
f = \sum_{g \in \Gamma} f_g, \ \text{ where} \ f_g(x) = \int_{\psi \in \hat{\Gamma}} \overline{\psi(g)}f(i(\psi)x) d\psi,
\end{equation*}
where the integration is with respect to the normalized Haar measure on $\hat{\Gamma}$. This decomposition holds as well for any $f \in C(G)$, at least if the convergence of $\sum_{g \in \Gamma} f_g$ is understood in $L^2(G)$.

Assume next we are given an invariant cocentral action of $\Gamma$ on $\Pol(G)$. It must be given by a continuous action $\alpha\colon\Gamma \curvearrowright G$ by group automorphisms leaving $i(\hat \Gamma)$ pointwise invariant. We denote by the same symbols $\alpha_g$ these automorphisms of $G$ and the corresponding automorphisms of~$C(G)$, so that $\alpha_g(f)(x)=f(\alpha^{-1}_g(x))$.

The coaction of $\Gamma$ on $\Pol(G) \rtimes_\alpha \Gamma$ defined in Proposition~\ref{prop:construction} can be regarded as an action of~$\hat{\Gamma}$. As $\Gamma$ is abelian, this is an action by $*$-algebra automorphisms. Explicitly, this action arises from the action of $\hat\Gamma$ on $C(G)$ by translations and from the dual action on $C(G)\rtimes_\alpha\Gamma$. We denote this combined action by $\tr\times\hat\alpha$, so
\begin{equation}\label{eg:hat-Gamma-action}
(\tr\times\hat\alpha)_\psi(f\lambda_g)=\overline{\psi(g)}f(i(\psi)\,\cdot)\lambda_g.
\end{equation}
Therefore Propositions~\ref{prop:construction} and~\ref{prop:reduced-algebra} imply that for this action of $\hat\Gamma$ on $C(G)\rtimes_\alpha\Gamma$ we have
$$
\Pol(G^{t,\alpha})=(\Pol(G)\rtimes_\alpha\Gamma)^{\hat\Gamma}\ \ \text{and}\ \ C(G^{t,\alpha})=(C(G)\rtimes_\alpha\Gamma)^{\hat\Gamma}.
$$
Note that as any compact group is coamenable, by Proposition~\ref{prop:enveloping} the quantum group $G^{t,\alpha}$ is coamenable as well, so the notation $C(G^{t,\alpha})$ is unambiguous.

Finally, the automorphisms $\alpha_g$ of $C(G)$ extend to $C(G)\rtimes_\alpha\Gamma$ by $\alpha_g(f\lambda_h)=\alpha_g(f)\lambda_h$. On~$\Pol(G^{t,\alpha})$, this is what we denoted by $\alpha_g\otimes\iota$ above.

With this setup, Proposition~\ref{prop:correspideals} (which is easily seen to remain true for Hopf $*$-algebras) translates into the following.

\begin{proposition}\label{prop:correspideals2}
Let $G$ be a compact group, $\Gamma$ a discrete abelian group, $i\colon\hat\Gamma\to Z(G)$ a continuous homomorphism, and $\alpha$ an action of $\Gamma$ on $G$ by group automorphisms that leave $i(\hat\Gamma)$ pointwise invariant, so that we can define the twisted compact quantum group~$G^{t,\alpha}$. Assume~$J$ is a Hopf $*$-ideal in $\Pol(G^{t,\alpha})$ such that $\alpha_g(J)=J$ for all $g\in\Gamma$ and $J$ is contained in the kernel of the homomorphism $\tilde p\colon \Pol(G)\rtimes_\alpha\Gamma\to\Pol(\hat\Gamma)$, $f\lambda_g\mapsto f (e)g$. Then~$J$ is defined by a closed $\alpha$-invariant subgroup $H$ of $G$ containing $i(\hat\Gamma)$, so that $\Pol(G^{t,\alpha})/J\simeq\Pol(H^{t,\alpha})$.
\end{proposition}

In order to formulate our main result we need to introduce more notation. Under the assumptions of the above proposition, define an action of $\Gamma \times \hat{\Gamma}$ on $G$ by homeomorphisms by
\begin{equation} \label{eq:double-action}
(g, \psi). x= \alpha_g(i(\psi)x)=i(\psi)\alpha_g(x).
\end{equation}
Consider the following subsets of $G$:
\begin{align*}
G_\reg &=\{x\in G\mid\text{the stabilizer of}\ x\ \text{in}\ \Gamma\times\hat\Gamma\ \text{is trivial}\},\\
G_1 &=\{ x \in G \mid \alpha_g(x) \in i(\hat{\Gamma})x, \ \forall g \in \Gamma\},\\
G^\Gamma &=\{x \in G \mid \alpha_g(x) = x, \ \forall g \in \Gamma \} \subset G_1.
\end{align*}

Our goal is to prove the following theorem.

\begin{theorem}\label{thm:subgroups}
Let $G$, $\Gamma$, $i$, $\alpha$ be as in Proposition~\ref{prop:correspideals2}, and assume that $\Gamma$ is finite and $i\colon\hat\Gamma\to Z(G)$ is injective. Assume also that one of the following conditions holds:
\begin{enumerate}
\item $\wedge^2\Gamma=0$ and $G = G_\reg \cup G_1$; or
\item $G = G_\reg \cup G^\Gamma$.
\end{enumerate}
Then there is a one-to-one correspondence between
\begin{enumerate}
\item[--] Hopf $*$-ideals $J\subset \Pol(G^{t,\alpha})$ with $\Pol(G^{t,\alpha})/J$ noncommutative, and
\item[--] closed $\alpha$-invariant subgroups of $G$ containing $i(\hat\Gamma)$ and an element of $G_{\rm reg}$.
\end{enumerate}
\end{theorem}

\begin{remark}
If $\Gamma$ is a cyclic group of prime order, then $G = G_\reg \cup G_1$, so condition (i) is satisfied.
\end{remark}

\subsection{Classification of irreducible representations}

The proof of Theorem~\ref{thm:subgroups} is based on a
classification of the irreducible representations of $C(G^{t,\alpha})$.  Let us record some C$^*$-algebraic facts that we will use.

\smallskip

When $A$ is a C$^*$-algebra, we denote the set of equivalence classes of its irreducible representations by $\hat{A}$.
We need a very particular case of the Mackey-type analysis of crossed product C$^*$-algebras $A=C(X)\rtimes\Gamma$~\cite{MR2288954}. Namely, let $X$ be a compact topological space, and $\Gamma$ be a finite abelian group acting on~$X$. Then the irreducible representations of $C(X) \rtimes \Gamma$ are parametrized by the pairs $([x], \psi)$, where $[x]$ is a class in $X/\Gamma$ and $\psi$ is a character on the stabilizer group $\St_x$ of $x$. The irreducible representation corresponding to such a pair is given as the induction of the representation $\ev_x \otimes \psi$ on $C(\Gamma x) \otimes C^*(\St_x)$ to $C(\Gamma x) \rtimes \Gamma$ composed with the restriction homomorphism $C(X) \rtimes \Gamma \to C(\Gamma x) \rtimes \Gamma$. In fact, all we need to know is that any irreducible representation of $C(X) \rtimes \Gamma$ factors through $C(\Gamma x) \rtimes \Gamma$ for some $x\in X$, which is just the first easy step in the Mackey-type analysis.

\smallskip

Recall also that if $A\subset B$ is an inclusion of C$^*$-algebras, then any irreducible representation of $A$ appears as a subrepresentation of some irreducible representation of $B$.

\smallskip

Turning to $C(G^{t,\alpha})$, from the above discussion we conclude that any irreducible representation of $C(G^{t,\alpha})$ factors through $C(G^{t,\alpha})\to C(\Gamma x)\rtimes\Gamma$, where $x \in G$ and $\Gamma x=\{\alpha_g(x)\}_{g\in\Gamma}$. Instead of the $\Gamma$-orbit of $x$ it is convenient to consider the larger $(\Gamma\times \hat\Gamma)$-orbit. The point is that the image of $C(G^{t,\alpha})$ in $C((\Gamma\times \hat\Gamma).x)\rtimes\Gamma$ coincides with the fixed point algebra $(C((\Gamma\times \hat\Gamma).x)\rtimes\Gamma)^{\hat\Gamma}$ with respect to the action of $\hat\Gamma$ defined by~\eqref{eg:hat-Gamma-action}. It follows that every irreducible representation~$\pi$ of~ $C(G^{t,\alpha})$ factors through $(C(O_\pi)\rtimes\Gamma)^{\hat\Gamma}$ for a uniquely defined $(\Gamma\times\hat\Gamma)$-orbit $O_\pi$ in $G$. Under the assumption of Theorem~\ref{thm:subgroups}, we just need to consider two types of orbits.

\smallskip

First consider a $(\Gamma\times\hat\Gamma)$-orbit $O$ in $G_\reg$. Then $O$ can be $(\Gamma\times\hat\Gamma)$-equivariantly identified with $\Gamma\times\hat\Gamma$. Therefore to understand the corresponding representations it suffices to observe the following.

\begin{lemma}
Consider the action of $\Gamma$ on $C(\Gamma\times\hat\Gamma)$ by translations and the action of $\hat \Gamma$ on $C(\Gamma\times\hat\Gamma)\rtimes\Gamma$ defined similarly to~\eqref{eg:hat-Gamma-action} as the combination of the action by translations on $C(\Gamma\times\hat\Gamma)$ and the dual action on $C^*(\Gamma)$.
Then $(C(\Gamma\times\hat\Gamma)\rtimes\Gamma)^{\hat\Gamma}\simeq \Mat_{|\Gamma|}(\C)$.
\end{lemma}

\bp Indeed, the algebra $C(\Gamma\times\hat\Gamma)\rtimes\Gamma$ is the direct sum of the blocks $C(\Gamma\times\{\psi\})\rtimes\Gamma\simeq  \Mat_{|\Gamma|}(\C)$ over $\psi\in\hat\Gamma$, while the action of $\hat\Gamma$ permutes these blocks.
\ep

Therefore every $(\Gamma\times\hat\Gamma)$-orbit $O$ in $G_\reg$ determines a unique up to equivalence irreducible representation of $C(G^{t,\alpha})$. Explicitly, for every $x\in G$ we have a representation of $\rho_x$ of $C(G) \rtimes \Gamma$ on~$\ell^2(\Gamma x)$ defined by
$$
\rho_x(f \lambda_g) \delta_{\alpha_h(x)} = f(\alpha_{g h}(x)) \delta_{\alpha_{g h}(x)}.
$$
If $x\in O$, then $\rho_x|_{C(G^{t,\alpha})}$ is the required representation corresponding to $O$, since it factors through
$(C(O)\rtimes\Gamma)^{\hat\Gamma}$ and has dimension $|\Gamma|$.

Let us also note that since the kernel of the map $C(G^{t,\alpha})\to (C(O)\rtimes\Gamma)^{\hat\Gamma}$ is $\alpha$- and $\hat\alpha$-invariant, the class of the representation corresponding to $O$ is $\alpha$- and $\hat\alpha$-invariant.

\smallskip

Next, consider an orbit $O$ in $G_1$. Fix a point $x\in O$. Then we have $\alpha_g(x)=i(B^x_g)x$ for a uniquely defined $B^x_g\in\hat\Gamma$. We thus get a homomorphism $B^x\colon\Gamma\to\hat\Gamma$. The bijective map $O=i(\hat\Gamma)x\to\hat\Gamma$, $i(\psi)x\mapsto \psi$, intertwines the action of $\Gamma\times\hat\Gamma$ on $i(\hat\Gamma)x$ with the action $(g,\psi)\eta=B^x_g\psi\eta$. Therefore it suffices to understand the algebra $(C(\hat\Gamma)\rtimes\Gamma)^{\hat\Gamma}$.

\begin{lemma}
Assume we are given a homomorphism $B\colon\Gamma\to\hat\Gamma$. Define an action $\beta$ of $\Gamma$ on~$C(\hat \Gamma)$ by $\beta_g(f)(\psi)=f(B_g^{-1}\psi)$, and consider the action $\tr\times\hat\beta$ of $\hat\Gamma$ on $C(\hat\Gamma)\rtimes_\beta\Gamma$ defined similarly to~\eqref{eg:hat-Gamma-action}. If either $\wedge^2\Gamma=0$ or $B$ is trivial, then the algebra $(C(\hat\Gamma)\rtimes_\beta\Gamma)^{\hat\Gamma}$ is abelian and the dual action $\hat\beta$ defines a free transitive action of $\hat\Gamma$ on its spectrum.
\end{lemma}

\bp If we view elements of $\Gamma$ as functions on $\hat\Gamma$, then $(C(\hat\Gamma)\rtimes_\beta\Gamma)^{\hat\Gamma}$ is spanned by the unitaries $u_g=g\lambda_g$. We have
$$
u_gu_h=g\beta_g(h)\lambda_{gh}=\overline{h(B_g)}u_{gh}.
$$
Therefore we see that $(C(\hat\Gamma)\rtimes\Gamma)^{\hat\Gamma}$ is a twisted group C$^*$-algebra of $\Gamma$. If $\wedge^2\Gamma=0$ or $B$ is trivial, then the corresponding $2$-cocycle is a coboundary, and $(C(\hat\Gamma)\rtimes_\beta\Gamma)^{\hat\Gamma}\simeq C^*(\Gamma)$ via an isomorphism $q$ that maps $u_g$ into a scalar multiple of $\lambda_g\in C^*(\Gamma)$. As $C^*(\Gamma)\simeq C(\hat\Gamma)$ is abelian, we get the first statement in the formulation, and since the isomorphism $q$ intertwines $\hat\beta$ with the standard dual action on $C^*(\Gamma)$, we get the second statement as well.
\ep

Thus, under the assumption of Theorem~\ref{thm:subgroups}, every orbit $O$ in $G_1$ (which is just $G^\Gamma$ for the case~(ii)) gives us a set of $|\Gamma|$ one-dimensional representations, on which~$\hat\Gamma$ acts transitively. Note again that since the kernel of the map $C(G^{t,\alpha})\to( C(O)\rtimes\Gamma)^{\hat\Gamma}$ is $\alpha$-invariant, the action of $\Gamma$ leaves this set invariant.

\smallskip

To summarize, we get the following classification.

\begin{proposition}\label{prop:rep-classification}
Under the assumption of Theorem~\ref{thm:subgroups}, consider the action of $\Gamma\times\hat\Gamma$ on~$\widehat{C(G^{t,\alpha})}$ defined by $\alpha$ and $\hat\alpha$, that is, $(g,\psi)[\pi]=[\pi\alpha_g^{-1}\hat\alpha_\psi^{-1}]$. Then the set $\widehat{C(G^{t,\alpha})}$ consists of points of two types:
\begin{enumerate}
\item every $(\Gamma\times\hat\Gamma)$-orbit in $G_\reg$ defines a point in $\widehat{C(G^{t,\alpha})}$ that is represented by a $|\Gamma|$-dimensional representation and is stabilized by $\Gamma\times\hat\Gamma$; namely, the point is represented by $\rho_x|_{C(G^{t,\alpha})}$ for any $x\in G_\reg$ in the original orbit;
\item every $\hat\Gamma$-orbit in $G_1$ defines a $\hat\Gamma$-orbit in $\widehat{C(G^{t,\alpha})}$ consisting of $|\Gamma|$ one-dimensional representations; this orbit is invariant under the action of $\Gamma$.
\end{enumerate}

In particular, $\Gamma[\pi]\subset\hat\Gamma[\pi]$ for any $[\pi]\in\widehat{C(G^{t,\alpha})}$, and the set of $\hat\Gamma$-orbits in $\widehat{C(G^{t,\alpha})}$ is canonically identified with the set of $(\Gamma\times\hat\Gamma)$-orbits in $G$.
\end{proposition}

\subsection{Classification of noncommutative quotients}

We still need some preparation in order to apply our description of irreducible representations to the classification of quantum subgroups.

\begin{lemma}
\label{lem:hopf-ideal-as-intersection-of-prim-ideals}
 Let $G$ be a compact quantum group and let $I$ be a Hopf $*$-ideal of $\Pol(G)$. Then
$$
I = \bigcap_{\rho  \in \widehat{C_u(G)},\ \rho(I)=0}\Ker(\rho),
$$
where $C_u(G)$ denotes the C$^*$-envelope of $\Pol(G)$.
\end{lemma}

\begin{proof}
It suffices to show that $\Pol(G)/I$ admits a faithful representation on a Hilbert space. But this is true, since $\Pol(G)/I$, being generated by matrix coefficients of finite dimensional unitary corepresentations, is the Hopf $*$-algebra of regular functions on a compact quantum group.
\end{proof}

For $T \in B(H)$, let $T^t$ be the transpose operator acting on $H^* \simeq \bar{H}$. In terms of the latter space, we have $T^t \bar{\xi} = \overline{T^* \xi}$. Then, whenever $\pi$ is a representation of $C(G) \rtimes_\alpha \Gamma$, we obtain another representation $\pi^\vee$ defined by $x \mapsto \pi^\vee(x) = \pi(S(x))^t$.

\begin{lemma}
\label{lem:eqv-rho-under-inverse}
For any $x\in G_\reg$ and $\psi\in\hat\Gamma$, the representation $\rho_{i(\psi)}|_{C(G^{t,\alpha})}$ is a subrepresentation of $(\rho_x\otimes\rho_x^\vee)|_{C(G^{t,\alpha})}$.
\end{lemma}

\bp Since the counit of $C(G) \rtimes_\alpha \Gamma$ coincides with $\rho_e$, the lemma is true for $\psi=e\in\hat\Gamma$. But since $\rho_{i(\psi)}=\rho_e\hat\alpha_\psi$, $\Delta\hat\alpha_\psi=(\hat\alpha_\psi\otimes\iota)\Delta$ and $\rho_x\sim \rho_x\hat\alpha_\psi$ by Proposition~\ref{prop:rep-classification}(i), the lemma is then true for any $\psi$.
\ep

\bp[Proof of Theorem~\ref{thm:subgroups}]
Assume that $J\subset\Pol(G^{t,\alpha})$ is a Hopf $*$-ideal with $\Pol(G^{t,\alpha})/J$ noncommutative. We want to apply Proposition~\ref{prop:correspideals2} and for this we have to show that $J$ is $\alpha$-invariant and contained in the kernel of the homomorphism $\tilde p\colon \Pol(G)\rtimes_\alpha\Gamma\to\Pol(\hat\Gamma)$. We start by establishing the second property.

\smallskip

First we note that on $\Pol(G^{t,\alpha})$ the homomorphisms $\Pol(G)\rtimes_\alpha\Gamma\to\Pol(\hat\Gamma)$ given by $f\lambda_g\mapsto f(e)g$ and $f\lambda_g\mapsto fi$ coincide. Therefore on $\Pol(G^{t,\alpha})$ the homomorphism $\tilde p$ can be identified with $\oplus_{\psi \in \hat{\Gamma}} \rho_{i(\psi)}$.

By Lemma~\ref{lem:hopf-ideal-as-intersection-of-prim-ideals}, there exists an irreducible representation $\pi$ of $C(G^{t,\alpha})$ such that $J\subset\Ker(\pi)$ and $\pi(C(G^{t,\alpha}))$ is noncommutative. Since $J$ is a Hopf ideal, we have $J\subset\Ker(\pi\otimes\pi^\vee)$. As~$\pi$ cannot be one-dimensional, by Proposition~\ref{prop:rep-classification} the representation $\pi$ must be equivalent to $\rho_x|_{C(G^{t,\alpha})}$ for some $x\in G_\reg$. But then $\big(\oplus_{\psi \in \hat{\Gamma}} \rho_{i(\psi)}\big)|_{C(G^{t,\alpha})}$ is a subrepresentation of $\pi\otimes\pi^\vee$ by Lemma~\ref{lem:eqv-rho-under-inverse}, hence $J\subset\Ker(\tilde p)$. It follows, in particular, that $J$ is invariant under the coaction of $\Gamma$, that is, under the dual action $\hat\alpha$ of $\hat\Gamma$.

\smallskip

In order to prove that $J$ is $\alpha$-invariant, by Lemma~\ref{lem:hopf-ideal-as-intersection-of-prim-ideals} it suffices to show that if $\pi$ is an irreducible representation of $C(G^{t,\alpha})$ with $J\subset\Ker(\pi)$, then $J\subset\Ker(\pi\alpha_g^{-1})$. But this is clear, since we already proved that $J$ is $\hat\Gamma$-invariant, and $\Gamma[\pi]\subset\hat\Gamma[\pi]$ by Proposition~\ref{prop:rep-classification}.

\smallskip

Therefore by Proposition~\ref{prop:correspideals2} the Hopf $*$-ideal $J$ is defined by a unique closed $\alpha$-invariant subgroup~$H$ of~$G$ containing $i(\hat\Gamma)$, so that $\Pol(G^{t,\alpha})/J=\Pol(H^{t,\alpha})$. By the above argument $\rho_x|_{\Pol(G^{t,\alpha})}$ factors through $\Pol(H^{t,\alpha})$. This is possible only when $x\in H$.
\ep

\begin{example}
As a first illustration of Theorem \ref{thm:subgroups}, let us give the description of  the nonclassical quantum subgroups of $\SU_{-1}(2)$, due to Podle\`s \cite{MR1331688}. Recall first from \cite{MR3340190} and Example~\ref{ex:orth} above that $\SU_{-1}(2)$ can be obtained from $\SU(2)$ by our general twisting procedure, via the $\mathbb Z_2$-action
$$
\begin{pmatrix}
   a & b \\
c & d
  \end{pmatrix} \mapsto \begin{pmatrix}
   a & -b \\
-c & d
  \end{pmatrix}.
$$
Therefore  the non-classical quantum subgroups of $\SU_{-1}(2)$
 correspond to the closed subgroups of $\SU(2)$ containing $\{\pm I_2\}$, stable under the previous $\mathbb Z_2$-action
and containing an element $\begin{pmatrix}
   a & b \\
c & d
  \end{pmatrix}$ with $a b c d\not=0$.
\end{example}

\begin{example} Let $\xi \in \mathbb T$ be a root of unity of order $N\geq 2$, and consider the compact quantum group
${\rm U}_\xi(2)$, which is the compact form of $\GL_{\xi,\xi^{-1}}(2)$~\cite{MR1432363}. Thus, $\Pol(\GL_{\xi,\xi^{-1}}(2)) = \Pol(\mathrm{U}_\xi(2))$ is generated as a $*$-algebra by the elements $(u_{i j})_{i,j=1}^2$ subject to the relations
$$
u_{i 2} u_{i 1} = \xi u_{i 1} u_{i 2}, \quad u_{2 k} u_{1 k} = \xi^{-1} u_{1 k} u_{2 k}, \quad u_{2 1} u_{1 2} = \xi^{-2} u_{1 2} u_{2 1}, \quad [u_{2 2}, u_{1 1}] = 0,
$$
and that the matrix $u = (u_{i j})_{i,j=1}^2 \in M_2(\Pol(\mathrm{U}_\xi(2))$ is unitary. It is well-known \cite{MR1432363}  that $\GL_{\xi,\xi^{-1}}(2)$ is a cocycle twisting of $\GL(2)$. One can also check that this cocycle arises from our recipe in Section~\ref{sec:cpt-q-grps}. To be specific, let $g$ be a generator of $\mathbb{Z}_N$, and consider the $\mathbb Z_N$-action
$$g\begin{pmatrix}
   a & b \\
c & d
  \end{pmatrix} = \begin{pmatrix}
   a & \xi^{-1}b \\
\xi c & d
  \end{pmatrix}.
$$
The Pontryagin dual of $\mathbb Z_N$ embeds into the center of $\mathrm{U}(2)$ in such a way that the defining representation $u$ of $\mathrm{U}(2)$ has degree $g$. Then it follows from the definition that the resulting twisting of ${\rm U}(2)$ is ${\rm U}_\xi(2)$. At $N=2$, i.e.,~$\xi=-1$, the description of the non-classical quantum subgroups of ${\rm U}_{-1}(2)$ is similar to the one given for $\SU_{-1}(2)$ in the previous example.

The subset ${\rm U}(2)^{\ZZ_N}$ consists of the subgroup of diagonal matrices. If we further assume that $N$ is odd, we have a decomposition ${\rm U}(2)={\rm U}(2)_{\rm reg}\cup {\rm U}(2)^{\ZZ_N}$. Thus, in this case Theorem \ref{thm:subgroups} implies that the non-classical quantum subgroups of ${\rm U}_\xi(2)$ correspond to the closed subgroups of ${\rm U}(2)$ stable under the operations
$$\begin{pmatrix}
   a & b \\
c & d
  \end{pmatrix} \mapsto \begin{pmatrix}
   \xi^k a & \xi^{k-l}b \\
\xi^{k+l}c & \xi^kd
  \end{pmatrix}, \ k,l \in \mathbb Z$$
and containing a nondiagonal element.
\end{example}

\begin{remark}
Note that although ${\rm U}_\xi(2)$ is a cocycle twisting of ${\rm U}(2)$ in the usual sense,
the techniques of ~\cite{MR3194750}  do not work well for this example, because $C({\rm U}_\xi(2))$ does not have any irreducible representation of dimension $N^2$ (the order of the twisting subgroup $\mathbb Z_N \times \mathbb Z_N$), which would be required to apply
Theorem 3.6 in \cite{MR3194750}, and here our present graded twisting procedure is better adapted to the study of quantum subgroups. In yet another direction, note as well that ~\cite{MR2671812} (where a different convention for parameter~$\alpha$ is used than in~\cite{MR1432363}) also studies quantum subgroups of $\GL_{\alpha,\beta}(n)$ under certain restrictions on $(\alpha, \beta)$, but the framework there is strictly that of non semisimple (and hence non compact) quantum groups, and these results do not apply to ${\rm U}_\xi(2)$.
\end{remark}

\begin{example}
Let $\tau=(\tau_1, \ldots ,\tau_{n-1}) \in \mu_n^{n-1}$ and consider
 the quantum group  $\SU^\tau(n)$ defined in \cite{MR3340190}, corresponding to the graded twisting of $\SU(n)$ associated to the cocentral $\mathbb Z_n$-action given by the embedding of $\mathbb Z_n$ into the center of $\SU(n)$ and by
$\alpha(u_{ij})=\gamma_i\gamma_j^{-1}u_{ij}$, where $\gamma_i=\prod_{1\leq r<i}\tau_r$.
It follows from Theorem \ref{thm:subgroups} that if $n\geq 3$ is prime,  the non-classical quantum subgroups of  $\SU^\tau(n)$
correspond to the closed subgroups of $\SU(n)$ stable under the operations
$$(g_{ij}) \mapsto (\xi^k(\gamma_i\gamma_j^{-1})^lg_{ij}), \ k,l \in \mathbb Z$$
where $\xi$ is a root of unity of order $n$, and containing an element $g=(g_{ij})$ with $\gamma_i\gamma_j^{-1} \not=1$ for some indices $i,j$.

\end{example}

\raggedright

\bigskip

\end{document}